\documentclass[12pt]{amsart}
\usepackage{amsmath, amsthm, amssymb}
\usepackage[initials]{amsrefs}
\usepackage{tikz, xcolor}
\title{Ranks based on strong amalgamation Fra\"{i}ss\'{e} classes}
\author{Vincent Guingona}
\author{Miriam Parnes}
\address{Towson University}
\email{vguingona@towson.edu}
\urladdr{https://tigerweb.towson.edu/vguingona/}
\email{mparnes@towson.edu}
\date{\today}
\thanks{Special thanks to C. D. Hill, D. Ulrich, and the anonymous reviewer. \\ \indent 2010 \emph{Mathematics Subject Classification}. 03C45.}

\def\papermode{1}

\newcommand{\analgtriv}{a strong amalgamation }
\newcommand{\algtriv}{strong amalgamation }
\newcommand{\AlgTriv}{Strong Amalgamation }

\newtheorem{thm}{Theorem}[section]
\newtheorem{cor}[thm]{Corollary}
\newtheorem{lem}[thm]{Lemma}
\newtheorem{prop}[thm]{Proposition}

\newtheorem*{claim*}{Claim}
\newtheorem*{caseone}{Case 1}
\newtheorem*{casetwo}{Case 2}

\newtheorem{ques}[thm]{Question}

\theoremstyle{remark}
\newtheorem{rem}[thm]{Remark}
\newtheorem{expl}[thm]{Example}

\theoremstyle{definition}
\newtheorem{defn}[thm]{Definition}

\newcommand{\tp}{\mathrm{tp} }
\newcommand{\qftp}{\mathrm{qftp} }
\newcommand{\concat}{{}^{\frown} }

\newcommand{\sig}{\mathrm{sig} }
\newcommand{\acl}{\mathrm{acl} }

\newcommand{\fraisse}{Fra\"{i}ss\'{e} }
\newcommand{\KK}{\mathbf{K} }
\newcommand{\LO}{\mathbf{LO} }
\newcommand{\GG}{\mathbf{G} }
\newcommand{\HH}{\mathbf{H} }
\newcommand{\EE}{\mathbf{E} }
\newcommand{\BS}{\mathbf{S} }
\newcommand{\TT}{\mathbf{T} }
\newcommand{\Mon}{\mathfrak{C} }

\newcommand{\cC}{\mathcal{C} }
\newcommand{\cD}{\mathcal{D} }
\newcommand{\cS}{\mathcal{S} }
\newcommand{\cG}{\mathcal{G} }
\newcommand{\cGs}{\mathcal{G}_{\mathrm{s}} }
\newcommand{\cGns}{\mathcal{G}_{\mathrm{ns}} }
\newcommand{\arity}{\mathrm{arity} }
\newcommand{\oa}{\overline{a} }
\newcommand{\ob}{\overline{b} }
\newcommand{\oc}{\overline{c} }
\newcommand{\od}{\overline{d} }
\newcommand{\oj}{\overline{j} }
\newcommand{\ox}{\overline{x} }
\newcommand{\oy}{\overline{y} }
\newcommand{\oz}{\overline{z} }
\newcommand{\ee}{\overline{e} }
\newcommand{\ozero}{\overline{0} }
\newcommand{\Rk}{\mathrm{Rk} }
\newcommand{\dpRk}{\mathrm{dpRk} }
\newcommand{\opDim}{\mathrm{opDim} }
\newcommand{\lex}{\mathrm{lex} }
\newcommand{\Pow}{\mathcal{P} }
\newcommand{\image}{\mathrm{im} }
\newcommand{\sstar}{*}
\newcommand{\selfsim}{definably self-similar}
\newcommand{\generic}{fully relational}
\newcommand{\mathiff}{\Longleftrightarrow }
\newcommand{\condiff}{\text{iff } }

\begin{document}

\begin{abstract}
 In this paper, we introduce the notion of $\KK$-rank, where $\KK$ is \analgtriv \fraisse class.  Roughly speaking, the $\KK$-rank of a partial type is the number ``copies'' of $\KK$ that can be ``independently coded'' inside of the type.  We study $\KK$-rank for specific examples of $\KK$, including linear orders, equivalence relations, and graphs.  We discuss the relationship of $\KK$-rank to other ranks in model theory, including dp-rank and op-dimension (a notion coined by the first author and C. D. Hill in previous work).
\end{abstract}

\maketitle

%%%%%%%%%%%%%%%%%%%%%%%%%%%%%
%% Section -- Introduction %%
%%%%%%%%%%%%%%%%%%%%%%%%%%%%%

\section{Introduction}

In model theory, there are many different notions of dimension and rank that are used to measure the complexity of partial types in first-order theories.  Some of these notions of rank involve measuring the largest ``size'' of a certain combinatorial configuration that exists in the type.  For example, the dp-rank of a partial type is the largest depth of an ICT-pattern in the type (see Definition \ref{defn_dprank}).  Ideally, one would like a general framework that simultaneously captures various combinatorial notions of rank together in a single unified notion.

In a modest step towards that goal, we introduce a novel class of ranks we call \emph{$\KK$-rank} for \analgtriv \fraisse class, $\KK$.  The idea is to concretely codify the notion of a ``combinatorial configuration'' by using $\KK$-configurations (see Definition \ref{Defn_KKConfig}).  Roughly speaking the $\KK$-rank of a partial type counts the maximum number of ``copies'' of $\KK$ that can be ``independently coded'' in the type.  More formally, ``coding'' is captured by the notion of $\KK$-configuration and the number of ``copies'' is captured by iterative free superpositions (see Definition \ref{Defn_StarOp}), which leads to the notion of $\KK$-rank (see Definition \ref{Defn_KKRank}).

In some instances, $\KK$-rank does generalize known notions of model theoretic rank.  For example, if $\KK$ is the class of finite linear orders, then $\KK$-rank (linear order rank) generalizes dp-rank in distal theories (see Example \ref{Expl_LONIP}).  This is a consequence of the fact that linear order rank generalizes op-dimension (see Definition \ref{Defn_OPDim}) on theories without the independence property (i.e., NIP theories); see Proposition \ref{Prop_NIPLOopdim}.  The notion of op-dimension was first introduced by the first author and C. D. Hill in \cite{gh2} and has since been utilized in other model theoretic studies; e.g., \cite{sim}.  As another example, if $\KK$ is the class of all finite sets with a single equivalence relation, then $\KK$-rank is bounded by the dp-rank; see Proposition \ref{Prop_EEUpperBound}.  Both dp-rank and op-dimension are additive \cites{gh2, kou} (see Definition \ref{Defn_Additive}); we examine under what conditions, on both $\KK$ and the target theory, $\KK$-rank is additive.  From this analysis on the specific class $\KK$ of all finite linear orders, we derive a result which may be of independent interest: We give a new characterization of NIP for certain theories based on the growth rate of $\KK$-rank; see Theorem \ref{Thm_NIPLinearQuadLO}.

The original idea for the ``coding'' part of this work comes from a paper by the first author and Hill \cite{gh}, where they study a related notion called positive local combinatorial dividing lines.  The requirements on the \fraisse classes considered in that paper are more stringent; specifically, they are required to be indecomposable (see Section 1 of \cite{gh}).  In the current paper, we do not make this assumption.  The notion of ``coding'' in this manner is also related to the phenomenon of non-collapsing generalized indiscernibles, studied by the first author, Hill, and L. Scow in \cite{ghs}; a detailed explanation of this relationship may be found in Section 3 of \cite{gh}.  When building generalized indiscernibles indexed by a class $\KK$, one needs to assume $\KK$ has the Ramsey property.  However, certain useful uniformity aspects of indiscernibility may still exist in the absence of the Ramsey property.  First, we develop a ``pseudo-indiscerniblity'' when the index class is ``nice'' (see Proposition \ref{Prop_WeakGoodEmbedding} and Proposition \ref{Prop_GoodEmbedding}).  We then utilize this in type-counting arguments (e.g., Proposition \ref{Prop_KKquad2}).

This paper is organized as follows:  In Section \ref{Sect_Prelim}, we introduce the notation used in the paper and cover some basic definitions.  In Section \ref{Sect_ATFC}, we discuss the relevant concepts surrounding \algtriv \fraisse classes.  Primarily, we discuss the notion of free superposition, which formalizes the idea of ``independent copies'' of a \fraisse class.  In Section \ref{Sect_Config}, we study the notion of configurations, which formalizes the idea of ``coding'' a \fraisse class into a partial type.  In Section \ref{Sect_DividingLines}, we connect the work in this paper back to the dividing lines considered in \cite{gh}.  In particular, we discuss an interesting generalization of a few results from that paper.  In Section \ref{Sect_Ranks}, we define and examine $\KK$-ranks for various \algtriv \fraisse classes, $\KK$.  In Subsection \ref{Subsect_LOR}, we study $\KK$-rank where $\KK$ is the class of all finite linear orders, in Subsection \ref{Subsect_EQR}, we study $\KK$-rank where $\KK$ is the class of all finite sets with a single equivalence relation, and in Subsection \ref{Subsect_GR}, we study $\KK$-rank where $\KK$ is the class of all finite graphs.  We study each of these $\KK$-ranks for types in the theory of the random graph in Subsection \ref{Subsect_RandGraph} and, in Subsection \ref{Subsect_RankDivide}, we explore the additivity of some ranks in theories without the independence property.  Finally, in Section \ref{Sect_Conclusion}, we discuss some interesting open problems.

%%%%%%%%%%%%%%%%%%%%%%%%%%%%%%
%% Section -- Preliminaries %%
%%%%%%%%%%%%%%%%%%%%%%%%%%%%%%

\section{Preliminaries}\label{Sect_Prelim}

Let $L$ be a first-order language.  By the \emph{signature} of $L$, denoted $\sig(L)$, we mean the set of constant symbols, function symbols, and relation symbols used in $L$.  We say $L$ is \emph{finite relational} if $\sig(L)$ is finite and consists only of relation symbols.  For a relation symbol $R \in \sig(L)$, we denote the arity of $R$ by $\arity(R)$.  If $M$ is an $L$-structure and $A \subseteq M$, we let $L(A)$ denote the language that expands $L$ by adding constant symbols to the signature for each $a \in A$.  Abusing notation, also let $L(A)$ also denote the set of $L(A)$-formulas.  For languages $L$ and $L_0$, we say $L_0$ is a \emph{reduct} of $L$ if $\sig(L_0) \subseteq \sig(L)$.  If $M$ is an $L$-structure and $L_0$ is a reduct of $L$, let $M |_{L_0}$ denote the reduct of $M$ to $L_0$.  If $M$ and $N$ are $L$-structures, write $M \cong_L N$ to mean that $M$ and $N$ are isomorphic as $L$-structures (where we drop the $L$ if it is clear).

In this paper, we will often be working with two first-order theories in different languages simultaneously, the ``index'' theory and the ``target'' theory.  Typically, the index theory will be the \fraisse limit of \analgtriv \fraisse class (see Definition \ref{Defn_ATFC}) in a finite relational language and the target theory will be an arbitrary complete first-order theory in an arbitrary language.

On the target side, suppose that $T$ is a complete first-order theory in a language $L$.  We use $\Mon$ to denote the \emph{monster model} of $T$; in this paper, it suffices to take $\Mon$ to be any model of $T$ that is at least $\aleph_1$-saturated.  We will also consider partial types $\pi(\oy)$, which are consistent collections of $L(A)$-formulas with free variables $\oy$ for some $A \subseteq \Mon$.  In this paper, we only consider such partial types over a \emph{small} $A$ (i.e., $A$ is smaller than the saturation of $\Mon$).  For a partial type $\pi$ and $M$ a substructure of $\Mon$, let $\pi(M)$ denote the set of all realizations of $\pi$ from $M$.  If $\varphi$ is a formula and $P$ is some property, then we write $\varphi^{\condiff P}$ to denote the formula $\varphi$ if $P$ is true and $\neg \varphi$ if $P$ is false.  If $t < 2$, we will write $\varphi^t$ to denote $\varphi^{\condiff t = 1}$.

For the following two definitions, let $T$ be a complete, first-order theory in a language $L$, let $\Mon$ be a monster model of $T$, and let $\pi(\oy)$ be a partial type.  We define two notions of rank that we will consider in this paper, \emph{dp-rank} and \emph{op-dimension}.  For simplicity of presentation, we will only consider $\omega$-valued dp-ranks and op-dimensions (generally, these can be defined to be ordinal-valued).

\begin{defn}\label{defn_dprank}
 Let $m < \omega$ and $\beta$ be an ordinal.  We say that $\pi$ has an \emph{ICT-pattern} of depth $m$ and length $\beta$ if there exist $L(\Mon)$-formulas $\varphi_i(\oy, \oz_i)$ for $i < m$ and $\oc_{i,j} \in \Mon^{|\oz_i|}$ for $i < m$ and $j < \beta$ such that, for all $g : m \rightarrow \beta$, the partial type
 \[ 
  \pi(\oy) \cup \{ \varphi_i(\oy, \oc_{i,j})^{\condiff g(i) = j} : i < m, j < \beta \}
 \]
 is consistent.  The \emph{dp-rank} of $\pi$ is the maximum $m < \omega$ such that $\pi$ has an ICT-pattern of depth $m$ and length $\omega$.  We denote the dp-rank of $\pi$ by $\dpRk(\pi)$.
\end{defn}

\begin{defn}\label{Defn_OPDim}
 Let $m < \omega$ and $\beta$ be an ordinal.  We say that $\pi$ has an \emph{IRD-pattern} of depth $m$ and length $\beta$ if there exist $L(\Mon)$-formulas $\varphi_i(\oy, \oz_i)$ for $i < m$ and $\oc_{i,j} \in \Mon^{|\oz_i|}$ for $i < m$ and $j < \beta$ such that, for all $g : m \rightarrow \beta$, the partial type
 \[ 
  \pi(\oy) \cup \{ \varphi_i(\oy, \oc_{i,j})^{\condiff g(i) < j} : i < m, j < \beta \}
 \]
 is consistent.  The \emph{op-dimension} of $\pi$ is the maximum $m < \omega$ such that $\pi$ has an IRD-pattern of depth $m$ and length $\omega$.  We denote the op-dimension of $\pi$ by $\opDim(\pi)$.
\end{defn}

In this paper, we attempt to generalize ``combinatorial patterns,'' like ICT-patterns or IRD-patterns, in order to define a generalized notion of rank.  To do this, we view the patterns as coming from an ``index'' theory.

On the index side, let $L_0$ be a finite relational language and let $\KK$ be a class of finite $L_0$-structures closed under isomorphism.
\begin{itemize}
 \item We say that $\KK$ has the \emph{hereditary property} if, for all $B \in \KK$ and $A \subseteq B$, $A \in \KK$.
 \item We say that $\KK$ has the \emph{joint embedding property} if, for all $A_0, A_1 \in \KK$, there exist $B \in \KK$ and embeddings $f_t : A_t \rightarrow B$ for each $t < 2$.
 \item We say that $\KK$ has the \emph{amalgamation property} if, for all $A, B_0, B_1 \in \KK$ and embeddings $f_t : A \rightarrow B_t$ for each $t < 2$, there exist $C \in \KK$ and embeddings $g_t : B_t \rightarrow C$ such that $g_0 \circ f_0 = g_1 \circ f_1$.
 \item We say that $\KK$ is a \emph{\fraisse class} if it has the hereditary property, the joint embedding property, and the amalgamation property.
\end{itemize}
The \emph{\fraisse limit} of a \fraisse class $\KK$ is the unique (up to isomorphism) countable $L_0$-structure $\Gamma$ such that $\Gamma$ is ultrahomogeneous and $\KK$ is the class of all finite structures embeddable into $\Gamma$ (see Theorem 6.1.2 of \cite{littlehodges}).  Since $L_0$ is a finite relational language, the theory of the \fraisse limit of $\KK$ is $\aleph_0$-categorical and eliminates quantifiers (see Theorem 6.4.1 of \cite{littlehodges}).  Abusing terminology, we will typically say that $\KK$ has a certain property if its \fraisse limit does.  This allows us to avoid writing the phrase ``whose \fraisse limit satisfies'' throughout the paper.

In this paper, we will be interested in ``coloring properties'' of the limits of \fraisse classes.  The following two definitions can be found in, for example, \cite{ezs}; here, we have rephrased them to be about colorings of the \fraisse limit.

\begin{defn}\label{Defn_Indivisible}
 Let $\KK$ be a \fraisse class with \fraisse limit $\Gamma$.  We say that $\KK$ is \emph{indivisible} if, for all $k < \omega$ and $c : \Gamma \rightarrow k$, there exist $\Gamma' \subseteq \Gamma$ with $\Gamma' \cong \Gamma$ and $i < k$ such that
 \[
  c(\Gamma') = \{ i \}.
 \]
\end{defn}

\begin{defn}\label{Defn_AgeIndivisible}
 Let $\KK$ be a \fraisse class with \fraisse limit $\Gamma$.  We say that $\KK$ is \emph{age indivisible} if, for all $k < \omega$, all $c : \Gamma \rightarrow k$, and all $A \in \KK$, there exist an embedding $f : A \rightarrow \Gamma$ and $i < k$ such that $c(f(A)) = \{ i \}$.
\end{defn}

It is clear that indivisibility implies age indivisibility.

In the next section, we will study these coloring properties in the context of \algtriv \fraisse classes.

%%%%%%%%%%%%%%%%%%%%%%%%%%%%%%%%%%%%%%%%%%%%%%%%%%%%
%% Section -- Strong Amalgamation Fraisse Classes %%
%%%%%%%%%%%%%%%%%%%%%%%%%%%%%%%%%%%%%%%%%%%%%%%%%%%%

\section{\AlgTriv \fraisse Classes}\label{Sect_ATFC}

In this section, we define the notion of \algtriv \fraisse class.  We then explore the free superposition and its relationship to some properties of \algtriv \fraisse classes.

Fix $L_0$ a finite relational language and let $\KK$ be a \fraisse class in $L_0$.  We can strengthen the amalgamation property as follows: We say $\KK$ satisfies the \emph{strong amalgamation property} if, for all $A, B_0, B_1 \in \KK$ and embeddings $f_t : A \rightarrow B_t$ for each $t < 2$, there exist $C \in \KK$ and embeddings $g_t : B_t \rightarrow C$ such that $g_0 \circ f_0 = g_1 \circ f_1$ and $g_0(B_0) \cap g_1(B_1) = g_0(f_0(A))$.  Since the language is relational, we may assume that the empty structure is in $\KK$, so we obtain a ``strong'' joint embedding property from the strong amalgamation property.  Moreover, if $\Gamma$ is the \fraisse limit of $\KK$, then $\KK$ has the strong amalgamation property if and only if, for all $A \subseteq \Gamma$, $\acl(A) = A$; see (2.15) of \cite{cameron_1990}.

\begin{defn}\label{Defn_ATFC}
 Let $\KK$ be a \fraisse class in a finite relational language.  We say that $\KK$ is a \emph{\algtriv \fraisse class} if it satisfies the strong amalgamation property.
\end{defn}

For each $t < 2$, let $\KK_t$ be a class of finite $L_t$-structures, where $L_t$ is a finite relational language.  Let $L_2$ be the language whose signature is the disjoint union of the signatures of $L_0$ and $L_1$ and define the \emph{free superposition} of $\KK_0$ and $\KK_1$, denoted $\KK_0 \sstar \KK_1$, as the class of all finite $L_2$-structures $A$ such that $A |_{L_t} \in \KK_t$ for each $t < 2$.

\begin{rem}\label{Rem_Gluing}
 Suppose that $A \in \KK_0$, $B \in \KK_1$, and $f : A \rightarrow B$ is a bijection.  Then, we can ``glue'' $A$ and $B$ together via $f$ to make an element of $\KK_0 \sstar \KK_1$.  Formally, let $C$ be the $L_2$-structure with universe $A$ such that, for all $R \in \sig(L_2)$ and $\oa \in A^{\arity(R)}$,
 \begin{itemize}
  \item if $R \in \sig(L_0)$, then $C \models R(\oa) \mathiff A \models R(\oa)$, and
  \item if $R \in \sig(L_1)$, then $C \models R(\oa) \mathiff B \models R(f(\oa))$.
 \end{itemize}
 Then, clearly $C \in \KK_0 \sstar \KK_1$.  Indeed, $C |_{L_0} = A$ and $C |_{L_1} \cong_{L_1} B$.
\end{rem}

\begin{prop}[Lemma 3.22 of \cite{bod2}]\label{Prop_ATFCstar}
 If $\KK_0$ and $\KK_1$ are \algtriv \fraisse classes, then $\KK_0 \sstar \KK_1$ is \analgtriv \fraisse class.
\end{prop}

Although the result is known, we give a proof here, as it will help in the proof of Proposition \ref{Prop_SelfSimStar} below.

\begin{proof}
 We begin by exhibiting the strong amalgamation property.  Fix structures $A, B_0, B_1 \in \KK_0 \sstar \KK_1$ and fix $L_2$-embeddings $f_0 : A \rightarrow B_0$ and $f_1 : A \rightarrow B_1$.  In particular, $f_0$ and $f_1$ are both $L_s$-embeddings for each $s < 2$.  By the strong amalgamation property of $\KK_s$, there exist $C_s \in \KK_s$ and $L_s$-embeddings $g^s_t : B_t \rightarrow C_s$ for $t < 2$ such that $g^s_0 \circ f_0 = g^s_1 \circ f_1$ and $g^s_0(B_0) \cap g^s_1(B_1) = g^s_0(f_0(A))$.  By embedding into a larger structure and using the hereditary property, we may assume that $|C_0| = |C_1|$.  There exists a bijection $h : C_0 \rightarrow C_1$ such that the following diagram commutes:
 \begin{center}
  \begin{tikzpicture}
   \draw (0,0) node {$A$};
   \draw[->] (0.25,0.25) -- (0.5,0.5) node[anchor = south east] {$f_0$} -- (0.75,0.75);
   \draw (1,1) node {$B_0$};
   \draw[->] (0.25,-0.25) -- (0.5,-0.5) node[anchor = north east] {$f_1$} -- (0.75,-0.75);
   \draw (1,-1) node {$B_1$};
   \draw[->] (1.3,1) -- (2,1) node[anchor = south] {$g_0^0$} -- (2.7,1);
   \draw[->] (1.25,0.8) -- (2.75,-0.8);
   \draw (1.8,0.7) node {$g_0^1$};
   \draw[->] (1.3,-1) -- (2,-1) node[anchor = north] {$g_1^1$} -- (2.7,-1);
   \draw[->] (1.25,-0.8) -- (2.75,0.8);
   \draw (1.8,-0.6) node {$g_1^0$};
   \draw (3,1) node {$C_0$};
   \draw (3,-1) node {$C_1$};
   \draw[->] (3,0.7) -- (3,0) node[anchor = west] {$h$} -- (3,-0.7);
  \end{tikzpicture}
 \end{center}
 As in Remark \ref{Rem_Gluing}, endow $C_0$ with an $L_2$-structure via $h$ and call it $C_2$.
 
 To exhibit the hereditary property, fix $B \in \KK_0 \sstar \KK_1$ and let $A \subseteq B$.  In particular, $A |_{L_0}$ is a $L_0$-substructure of $B |_{L_0}$, so $A |_{L_0} \in \KK_0$.  Similarly, $A |_{L_1} \in \KK_1$.  Thus, $A \in \KK_0 \sstar \KK_1$.
\end{proof}

\begin{expl}
 Note that the strong amalgamation property is necessary to conclude that the free superposition is even a \fraisse class.  For example, for each $t < 2$, let $L_t$ be the language with one unary predicate, $P_t$, and let $\KK_t$ be the class of all $L_t$-structures where at most one element satisfies $P_t$.  This is clearly a \fraisse class, but does not have strong amalgamation.  On the other hand, $\KK_0 \sstar \KK_1$ is not a \fraisse class, as it fails joint embedding.  Let $A_0 = \{ a_0, a_1 \}$ where $P_0(a_0)$ and $P_1(a_1)$ and let $A_1 = \{ a_2 \}$ where $P_0(a_2)$ and $P_1(a_2)$.  Then, there exists no $B \in \KK_0 \sstar \KK_1$ which embeds $A_0$ and $A_1$ simultaneously.
\end{expl}

\begin{defn}\label{Defn_PropertiesofKK}
 Let $\KK$ be \analgtriv \fraisse class in $L_0$, $A \in \KK$, and $R$ a relation of $L_0$ with arity $n$.
 \begin{enumerate}
  \item We say $R$ is \emph{symmetric} on $A$ if, for all $a \in {}^n A$ and all $\sigma \in S_n$, if $A \models R(a)$, then $A \models R(a \circ \sigma)$.
  \item We say $R$ is \emph{trichotomous} on $A$ if, for all $a \in {}^n A$ such that $a(i) \neq a(j)$ for all $i < j < n$, there exists exactly one $\sigma \in S_n$ such that $A \models R(a \circ \sigma)$.
  \item We say $R$ is \emph{reflexive} on $A$ if, for all $a \in {}^n A$ such that $a(i) = a(j)$ for all $i < j < n$, $A \models R(a)$.
  \item We say $R$ is \emph{irreflexive} on $A$ if, for all $a \in {}^n A$ such that $a(i) = a(j)$ for some $i < j < n$, $A \models \neg R(a)$.
  \item If $n = 2$, we say $R$ is \emph{transitive} if, for all $a, b, c \in A$, if $A \models R(a,b) \wedge R(b,c)$, then $A \models R(a,c)$.
 \end{enumerate}
 We say $A$ has one of the above properties if, for all $R \in \sig(L_0)$, $R$ has that property on $A$.  We say $\KK$ has one of the above properties if, for all $A \in \KK$, $A$ has that property.
\end{defn}

\begin{prop}\label{Prop_StarProperties}
 Each of the properties in Definition \ref{Defn_PropertiesofKK} is closed under free superposition.
\end{prop}

\begin{proof}
 Any witness to the failure of one of these properties in $\KK_0 \sstar \KK_1$ reducts to a failure of the same property in either $\KK_0$ or $\KK_1$.
 \if\papermode0
 { \color{violet}
 
  \emph{More Detail:} For each $t < 2$, let $\KK_t$ be \analgtriv \fraisse class in the finite relational language $L_t$.  Let $\KK_2 = \KK_0 \sstar \KK_1$, where $\KK_2$ is in the language $L_2$, whose signature is the disjoint union of the signatures of $L_0$ and $L_1$.
 
  (Symmetric): Suppose $\KK_2$ is not symmetric, witnessed by $R \in \sig(L_2)$.  Thus, there exist $A \in \KK_2$, $a \in {}^{\arity(R)} A$, and $\sigma \in S_{\arity(R)}$ such that $A \models R(a)$ and $A \models \neg R(a \circ \sigma)$.  Let $t < 2$ be such that $R \in \sig(L_t)$ and let $A_t = A |_{L_t}$.  Then, $A_t \models R(a)$ and $A_t \models \neg R(a \circ \sigma)$.  Thus, $\KK_t$ is not symmetric.
 
  The remainder of the properties follow similarly.
 }
 \fi
\end{proof}

\begin{defn}
 We have a few \algtriv \fraisse classes that we examine in particular in this paper.
 \begin{enumerate}
  \item (Sets) Let $\BS$ denote the class of all finite $L_0$-structures where $L_0$ has empty signature.
  \item (Linear Orders) Let $\LO$ denote the class of all finite $L_0$-structures that are trichotomous, irreflexive, and transitive, where $L_0$ is a language with one binary relation symbol.
  \item (Equivalence Relations) Let $\EE$ denote the class of all finite $L_0$-structures that are symmetric, reflexive, and transitive, where $L_0$ is a language with one binary relation symbol.
  \item (Graphs) Let $\GG$ denote the class of all finite $L_0$-structures that are symmetric and irreflexive, where $L_0$ is a language with one binary relation symbol.
  \item (Hypergraphs) For $k \ge 2$, let $\HH_k$ denote the class of all finite $L_0$-structures that are symmetric and irreflexive, where $L_0$ is a language with one $k$-ary relation symbol.  Clearly $\GG = \HH_2$.
  \item (Tournaments) Let $\TT$ denote the class of all finite $L_0$-structures that are trichotomous and irreflexive, where $L_0$ is a language with one binary relation symbol.
 \end{enumerate}
\end{defn}

\begin{defn}\label{Defn_StarOp}
 Suppose $\KK$ is \analgtriv \fraisse class and fix $n \ge 1$.  Then, define $\KK^{\sstar n}$ recursively as follows:
 \begin{enumerate}
  \item $\KK^{\sstar 0} = \BS$,
  \item $\KK^{\sstar (n+1)} = \KK^{\sstar n} \sstar \KK$.
 \end{enumerate}
\end{defn}

\begin{expl}
 For any \algtriv \fraisse class $\KK$, notice that
 \[
  \BS \sstar \KK = \KK \sstar \BS = \KK.
 \]
 So, in particular, $\KK^{\sstar 1} = \KK$ and $\BS^{\sstar n} = \BS$ for all $n \ge 1$.
\end{expl}

\begin{expl}
 For all $n \ge 1$, $\LO^{\sstar n}$ is the class of all finite sets with $n$ linear orders.
\end{expl}

\begin{expl}\label{Expl_GeneralHypergraphs}
 In any finite relational language $L_0$ where all relations are at least binary, the class of $L_0$-hypergraphs, $\HH_{L_0}$, is the set of all finite $L_0$-structures that are symmetric and irreflexive.  By Proposition \ref{Prop_StarProperties},
 \[
  \HH_{L_0} = \HH_{k_0} \sstar \dots \sstar \HH_{k_{n-1}},
 \]
 where $k_0 \le \dots \le k_{n-1}$ list all the arities (with repetition) of the relation symbols in $L_0$.  By Proposition \ref{Prop_ATFCstar}, $\HH_{L_0}$ is \analgtriv \fraisse class.
\end{expl}

In the remainder of this section, we will introduce tools that will be used to compute $\KK$-rank for specific \algtriv \fraisse classes $\KK$ in Section \ref{Sect_Ranks}.  We use the following proposition to build substructures of \fraisse limits that are isomorphic to the original limit.

\begin{prop}\label{Prop_WeakHom}
 Suppose that $\Gamma$ is the \fraisse limit of $\KK$ and $\Gamma' \subseteq \Gamma$.  If, for all $A, B \in \KK$ with $A \subseteq B$ and $|B \setminus A| = 1$ and for all embeddings $f : A \rightarrow \Gamma'$, there exists an embedding $g: B \rightarrow \Gamma'$ extending $f$, then $\Gamma' \cong \Gamma$.
\end{prop}

\begin{proof}
 This follows from Lemma 6.1.4 of \cite{littlehodges}.
 \if\papermode0
 { \color{violet}
 
  \emph{More Detail:} To show that $\Gamma'$ is isomorphic to the \fraisse limit of $\KK$, we must show that $\Gamma'$ has the same age as $\Gamma$, $\Gamma'$ is countable, and that $\Gamma'$ is weakly-homogeneous (see Lemma 6.1.4 of \cite{littlehodges}).  Since $\Gamma' \subseteq \Gamma$, $\Gamma'$ is countable.  By assuming that $\KK$ contains the empty $L_0$-structure, it suffices to show that $\Gamma'$ is weakly-homogeneous.  Since we can extend any finite structure ``one element at a time,'' it suffices to do this with $B \subseteq C$ such that $|C \setminus B| = 1$.  This is precisely the hypothesis of the proposition.
  }
 \fi
\end{proof}

The following definition is made in a general context, but we will mostly be interested in the case where $\Gamma$ is the \fraisse limit of \analgtriv \fraisse class $\KK$.

\begin{defn}\label{Defn_SelfSim}
 Let $\Gamma$ be a structure in any language, $L$.  We say $\Gamma$ is (\emph{quantifier-free}) \emph{\selfsim} if, for any finite $A \subseteq \Gamma$ and any complete, non-algebraic (quantifier-free) $1$-$L$-type $p$ over $A$, $p(\Gamma)$ is isomorphic to $\Gamma$.
\end{defn}

Note that, when $\Gamma$ is the \fraisse limit of a \fraisse class in a finite relational language, by quantifier elimination, being quantifier-free \selfsim\ is equivalent to being \selfsim.

\begin{lem}\label{Lem_SelfSim}
 Let $\KK$ be \analgtriv \fraisse class in a finite relational language $L_0$ with \fraisse limit $\Gamma$.  Then, $\Gamma$ is \selfsim\ if and only if, for all $B, B', C \in \KK$ such that $B \subseteq B'$ and $|B' \setminus B| = 1$, for all $A \subseteq C$ and $p$ a complete, non-algebraic $1$-$L_0$-type over $A$, for all embeddings $f : B \rightarrow p(C)$, there exist $C' \in \KK$ with $C' \supseteq C$ and an embedding $f' : B' \rightarrow p(C')$ extending $f$.
\end{lem}

\begin{proof}
 ($\Rightarrow$): Assume that $\Gamma$ is \selfsim.  Fix $B$, $B'$, $C$, $A$, $p$, and $f$ as in the lemma.  Since $\Gamma$ is the \fraisse limit of $\KK$, we may assume that $C \subseteq \Gamma$.  Since $\Gamma$ is \selfsim, $p(\Gamma) \cong \Gamma$.  By assumption, $f(B) \subseteq p(\Gamma)$.  Since $p(\Gamma)$ is also the \fraisse limit of $\KK$, there exists an embedding $f' : B' \rightarrow p(\Gamma)$ extending $f$.  Let $C' = f'(B') \cup C$.  This gives the desired extension.
 
 ($\Leftarrow$): Fix $A \subseteq \Gamma$ finite and $p$ a complete, non-algebraic $1$-$L_0$-type over $A$.  We show that the hypothesis of Proposition \ref{Prop_WeakHom} is satisfied for $p(\Gamma)$.  Consider $B, B' \in \KK$ with $B \subseteq B'$ and $|B' \setminus B| = 1$ and suppose that $f : B \rightarrow p(\Gamma)$ is an embedding.  Let $C = f(B) \cup A$, so $f$ is an embedding of $B$ into $p(C)$.  By assumption, there exists $C' \in \KK$ with $C' \supseteq C$ and an embedding $f' : B' \rightarrow p(C')$ extending $f$.  Since $\Gamma$ is the \fraisse limit of $\KK$, we may assume $C' \subseteq \Gamma$ and thus $f'$ embeds $B'$ into $p(\Gamma)$.
\end{proof}

The preceding lemma gives us a characterization of when the \fraisse limit of a \fraisse class is \selfsim\ in terms of the class.  Thus, we will say that $\KK$ is \emph{\selfsim} if its \fraisse limit is \selfsim.

\if\papermode0
 { \color{violet}
 
 \emph{More Detail:} We cannot hope to replace the definition of \selfsim\ by changing ``complete type over a finite set'' to ``definable formula.''  For example, consider the formula $\theta(x) = a \le x \le b$ in $\mathbb{Q}$ (for some $a, b \in \mathbb{Q}$ with $a < b$).  Then, even though $\mathbb{Q}$ (and hence $\LO$) is \selfsim, we don't have $\theta(\mathbb{Q}) \cong \mathbb{Q}$.
 }
\fi

The next lemma is essentially the same as Theorem 2.6 of \cite{sauer2020}, but we include a proof here for completeness.

\begin{lem}\label{Lem_Coloring}
 Suppose that $\KK$ is a \selfsim\ \algtriv \fraisse class.  Then, $\KK$ is indivisible.
\end{lem}

\begin{proof}
 Let $\Gamma$ be the \fraisse limit of $\KK$, let $k < \omega$, and let $c : \Gamma \rightarrow k$.  We find $\Gamma' \subseteq \Gamma$ with $\Gamma' \cong \Gamma$ and $i < k$ such that $c(\Gamma') = \{ i \}$.  We may assume $k = 2$.  Suppose that $c^{-1}( \{ 0 \} ) \not\cong \Gamma$.  Then, by the contrapositive of Proposition \ref{Prop_WeakHom}, there exist $A, B \in \KK$ with $A \subseteq B$ and $B = \{ b \} \cup A$, and $f: A \rightarrow c^{-1}( \{ 0 \} )$ an embedding that does not extend to an embedding of $B$ into $c^{-1}( \{ 0 \} )$.  Then, consider
 \[
  \Gamma' = \left\{ d \in \Gamma : \tp_{L_0}(d,f(A)) = \tp_{L_0}(b,A) \right\}.
 \]
 Since $\Gamma$ is \selfsim, $\Gamma' \cong \Gamma$.  On the other hand, for any $d \in \Gamma'$, the function extending $f$ to a function from $B$ to $\Gamma$ by sending $b$ to $d$ is an embedding.  Thus, $c(d) = 1$.  In other words, $c(\Gamma') = \{ 1 \}$.
\end{proof}

In the next proposition, we show that being \selfsim\ is closed under free superposition.

\begin{prop}\label{Prop_SelfSimStar}
 Suppose that $\KK_0$ and $\KK_1$ are \selfsim.  Then, $\KK_0 \sstar \KK_1$ is \selfsim.
\end{prop}

\begin{proof}
 Let $L_0$ be the language of $\KK_0$, let $L_1$ be the language of $\KK_1$, and let $L_2$ be the language whose signature is the disjoint union of the signatures of $L_0$ and $L_1$, which serves as the language for $\KK_2 = \KK_0 \sstar \KK_1$.
 
 We use the characterization in Lemma \ref{Lem_SelfSim}.  Fix $B, B', C \in \KK_2$ such that $B \subseteq B'$, fix $A \subseteq C$, fix $p(x)$ a complete, non-algebraic $1$-$L_2$-type over $A$, and fix an $L_2$-embedding $f : B \rightarrow p(C)$.  In particular, for each $t < 2$, $f$ is an $L_t$-embedding of $B$ into $p |_{L_t}(C)$.  For each $t < 2$, since $\KK_t$ is \selfsim, there exist $C'_t \in \KK_t$ with $C |_{L_t} \subseteq C'_t$ and an $L_t$-embedding $f'_t : B' \rightarrow p |_{L_t}(C'_t)$.  Using the hereditary property, we may assume that $|C'_0| = |C'_1|$.  Then, as in the proof of Proposition \ref{Prop_ATFCstar}, there exists a bijection $g$ from $C'_0$ to $C'_1$ such that the following diagram commutes:
 \begin{center}
  \begin{tikzpicture}
   \draw (0,0) node {$B$};
   \draw[->] (0.25,0.25) -- (0.5,0.5) node[anchor = south east] {$\iota$} -- (0.75,0.75);
   \draw (1,1) node {$B'$};
   \draw[->] (0.25,-0.25) -- (0.5,-0.5) node[anchor = north east] {$f$} -- (0.75,-0.75);
   \draw (1,-1) node {$C$};
   \draw[->] (1.3,1) -- (2,1) node[anchor = south] {$f'_0$} -- (2.7,1);
   \draw[->] (1.25,0.8) -- (2.75,-0.8);
   \draw (1.8,0.7) node {$f'_1$};
   \draw[->] (1.3,-1) -- (2,-1) node[anchor = north] {$\iota$} -- (2.7,-1);
   \draw[->] (1.25,-0.8) -- (2.75,0.8);
   \draw (1.8,-0.6) node {$\iota$};
   \draw (3,1) node {$C'_0$};
   \draw (3,-1) node {$C'_1$};
   \draw[->] (3,0.7) -- (3,0) node[anchor = west] {$g$} -- (3,-0.7);
  \end{tikzpicture}
 \end{center} 
 As in Remark \ref{Rem_Gluing}, create the structure $C' \in \KK_2$ with universe $C'_0$ endowed with $L_2$-structure given by $g$.  Then, it is not hard to show that $f'_0$ is an $L_2$-embedding of $B'$ into $p(C')$.
\end{proof}

\begin{expl}\label{Expl_LOGGSelfSim}
 The classes $\LO$, $\GG$, and $\TT$ are \selfsim.  Moreover, for all $k \ge 2$, $\HH_k$ is \selfsim.  By Proposition \ref{Prop_SelfSimStar}, for all $n \ge 1$, $\LO^{\sstar n}$, $\GG^{\sstar n}$, and $\TT^{\sstar n}$ are \selfsim\ and, for all $k \ge 2$, $\HH_k^{\sstar n}$ is \selfsim.  On the other hand, $\EE$ is not \selfsim.
\end{expl}

\begin{proof}
 In the theory of dense linear orders, for any complete, non-algebraic $1$-type $p$ over a finite subset of $\mathbb{Q}$, $p(\mathbb{Q})$ is an open interval, which is clearly isomorphic to $\mathbb{Q}$.  Hence, $\LO$ is \selfsim.
 
 Next, consider $\HH_k$ for some $k \ge 2$ in the language $L_0$ with one $k$-ary relation symbol, $E$.  Fix $B, B', C \in \HH_k$ with $B \subseteq B'$, fix $A \subseteq C$, and fix $p$ a complete, non-algebraic $1$-$L_0$-type over $A$.  Suppose that $f : B \rightarrow p(C)$ is an embedding and that $B' = B \cup \{ b' \}$.  Create an $L_0$-structure $C'$ where $C' = C \cup \{ c' \}$ by setting, for all $\ob \in B^{k-1}$,
 \[
  C' \models E(c', f(\ob)) \mathiff B' \models E(b', \ob),
 \]
 and, for all $\oa \in A^{k-1}$,
 \[
  C' \models E(c', \oa) \mathiff E(x, \oa) \in p(x),
 \]
 and add no additional edges (except those necessary to create symmetry).  Finally, extend $f$ to $f'$ by setting $f'(b') = c'$.  It is easy to check that $f'$ is an embedding from $B'$ into $p(C')$.  A similar argument works for $\TT$, where the ``direction'' of each edge is determined by either $B'$ or $p$.
 
 Finally, consider the class $\EE$ in the language $L_0$ with one binary relation symbol, $E$, and let $\Gamma$ be the \fraisse limit of $\EE$.  Fix $a \in \Gamma$ and let $p(x)$ be the complete $1$-$L_0$-type over $\{ a \}$ extending $x \neq a \wedge E(x,a)$.  Clearly $p(\Gamma) \not\cong \Gamma$.
\end{proof}

The above proof for $\HH_k$ can be modified to show that, if $\KK$ has $3$-amalgamation (see \cite{hru}), then $\KK$ is \selfsim.  On the other hand, $\LO$ witnesses that the converse is false.

Although $\EE$ is not \selfsim, we can analyze $\EE^{\sstar m}$ for $m \ge 1$.  The following lemma aids in this analysis.

\begin{lem}\label{Lem_EEmuniverse}
 For $m < \omega$, let $L_m$ be the language with $m$ binary relation symbols $E_i$ for $i < m$.  For any set $I$, we can put an $L_m$-structure on $I^{m+1}$ by setting, for all $i < m$ and all $\oa, \ob \in I^{m+1}$, $E_i(\oa, \ob)$ if $a_i = b_i$.  With this $L_m$-structure:
 \begin{enumerate}
  \item if $I$ is finite, then $I^{m+1} \in \EE^{\sstar m}$; and
  \item if $I$ is countably infinite, then $I^{m+1}$ is isomorphic to the \fraisse limit of $\EE^{\sstar m}$.
 \end{enumerate}
\end{lem}

\begin{proof}
 Trivial.
\end{proof}

\begin{expl}\label{Expl_EE2color}
 Although $\EE$ is indivisible, for $m \ge 2$, $\EE^{\sstar m}$ is not indivisible.
\end{expl}

\begin{proof}
 It follows from the Pigeonhole Principle that $\EE$ is indivisible.

 Fix $m \ge 2$ and let $\Gamma$ be the \fraisse limit of $\EE^{\sstar m}$.  By Lemma \ref{Lem_EEmuniverse}, we can suppose $\Gamma$ has universe $\omega^{m+1}$.  Consider the coloring $c : \Gamma \rightarrow 2$ given by
 \[
  c(\oa) = \begin{cases} 0 & \text{ if } a_0 < a_1, \\ 1 & \text{ if } a_0 \ge a_1 \end{cases}.
 \]
 Towards a contradiction, suppose there is $\Gamma' \cong \Gamma$ with $c(\Gamma') = \{ t \}$.  For any $\oa \in \Gamma'$, there are only finitely many $E_t$-classes in $\Gamma'$ that have non-empty intersection with the $E_{1-t}$-class of $\oa$ in $\Gamma'$, a contradiction.
 \if\papermode0
  { \color{violet}
  
   More Details: If $t = 0$, then $a_0 < a_1$, since $c(\oa) = 0$.  If $\ob \in \Gamma'$ is such that $b_1 = a_1$, then $b_0 < b_1 = a_1$, so there are only finitely many such choices for $b_0$.  Thus, there are only finitely many $E_0$-classes in $\Gamma'$ that intersect with the $E_1$-class of $\oa$ in $\Gamma'$.  This works similarly for $t = 1$.
  }
 \fi
\end{proof}

To deal with $\EE^{\sstar m}$ in Section \ref{Sect_Ranks}, we need a condition that is weaker than being \selfsim, but which is still strong enough to run counting arguments.  It turns out that age indivisibility is a sufficient condition for our purposes.  By Lemma \ref{Lem_Coloring}, if $\KK$ is \selfsim, then $\KK$ is age indivisible.

\begin{expl}\label{Expl_EEWeakSelfSim}
 For all $m < \omega$, $\EE^{\sstar m}$ is age indivisible.  In particular, age indivisibility is strictly weaker than being \selfsim.
\end{expl}

\begin{proof}
 Let $\Gamma$ be the \fraisse limit of $\EE^{\sstar m}$.  By Lemma \ref{Lem_EEmuniverse}, we may assume that $\Gamma$ has universe $\omega^{m+1}$.  Let $c : \Gamma \rightarrow k$ be a coloring and let $A \in \EE^{\sstar m}$.  Let $n = |A|$.  By Corollary \ref{Cor_ColorBoxes}, there exist $Y_0, \dots, Y_m \in \binom{\omega}{n}$ such that $c$ is constant on $B = \prod_{i \le m} Y_i$.  On the other hand, there is clearly an embedding $g : A \rightarrow B$.
 \if\papermode0
 { \color{violet}
 
  \emph{More Detail:} See the construction in the proof of Proposition \ref{Prop_EELowerBound}.
  }
 \fi
 Thus, $c$ is constant on $g(A)$.  This shows that $\EE^{\sstar m}$ is age indivisible.
\end{proof}

The property described in the following definition is mild, only requiring that any combination of relation symbols of the same arity that can happen, does happen.  However, it provides a lower bound for the number of types in our type-counting arguments in Section \ref{Sect_Ranks}.

\begin{defn}
 Let $\KK$ be \analgtriv \fraisse class.  We say that $\KK$ is \emph{\generic} if, for all $n < \omega$, for all functions $f$ from relation symbols in $L_0$ of arity $n$ to $2$, there exist $A \in \KK$ and $\oa \in A^n$ such that $a_i \neq a_j$ for all $i < j < n$ and, for all relation symbols $R$ in $L_0$ of arity $n$, $A \models R(\oa)$ if and only if $f(R) = 1$.
\end{defn}

Note that, for a language with a single $n$-ary relation symbol, being \generic\ means that there is one (non-repeating) $n$-tuple where the relation holds and one where it fails.

\begin{expl}
 Notice that $\BS$, $\LO$, $\EE$, $\GG$, $\TT$, and $\HH_k$ for all $k \ge 2$ are all \generic.
\end{expl}

\begin{prop}\label{Prop_GenericStar}
 Suppose that $\KK_0$ and $\KK_1$ are \generic.  Then, $\KK_0 \sstar \KK_1$ is \generic.
\end{prop}

\begin{proof}
 For each $t < 2$, let $\KK_t$ be \analgtriv \fraisse class in the finite relational language $L_t$ that is \generic.  Let $\KK_2 = \KK_0 \sstar \KK_1$, where $\KK_2$ is in the language $L_2$ whose signature is the disjoint union of the signatures of $L_0$ and $L_1$.  Fix $n < \omega$ and, for $t < 3$, let $\sig_n(L_t)$ denote the set of all relation symbols of $L_t$ of arity $n$.  Let $f : \sig_n(L_2) \rightarrow 2$.  For each $t < 2$, since $\KK_t$ is \generic, there exist $A_t \in \KK_t$ and $\oa^t \in A_t^n$ such that $a_i^t \neq a_j^t$ for all $i < j < n$ and, for all $R \in \sig_n(L_t)$, $A_t \models R(\oa^t)$ if and only if $f(R) = 1$.  By the hereditary property, we may assume that $A_t = \{ a_i^t : i < n \}$.  Then, let $A$ be the $L_2$-structure with universe $A_0$ induced by the bijection $a_i^0 \mapsto a_i^1$ as in Remark \ref{Rem_Gluing}.  Thus, $A \in \KK_2$ and, for all $R \in \sig_n(L_2)$, $A \models R(\oa_0)$ if and only if $f(R) = 1$.
\end{proof}

%%%%%%%%%%%%%%%%%%%%%%%%%%%%%%%
%% Section -- Configurations %%
%%%%%%%%%%%%%%%%%%%%%%%%%%%%%%%

\section{Configurations}\label{Sect_Config}

Throughout this section, let $L$ be any language, let $T$ be a complete $L$-theory, and let $\Mon$ be a monster model of $T$.

The following definition will be used to capture what we mean by ``coding'' the class $\KK$ in the partial type $\pi$.

\begin{defn}\label{Defn_KKConfig}
 Let $\KK$ be \analgtriv \fraisse class in a finite relational language $L_0$ and let $\pi(\oy)$ be a partial type in $T$.  A \emph{$\KK$-configuration} into $\pi$ is a family of functions $(I, f_A)_{A \in \KK}$ such that
 \begin{enumerate}
  \item $I : \sig(L_0) \rightarrow L(\Mon)$;
  \item for all $A \in \KK$, $f_A : A \rightarrow \pi(\Mon)$; and
  \item for all $R \in \sig(L_0)$, for all $A \in \KK$, for all $\oa \in A^{\arity(R)}$,
   \[
    A \models R(\oa) \mathiff \Mon \models I(R)(f_A(\oa)).
   \]
 \end{enumerate}
 Note that, for each $n$-ary relation symbol $R$ in $L_0$, the $L(\Mon)$-formula $I(R)$ has free variables consisting of an $n$-tuple of tuples of variables, each of the same sort as $\oy$.
 
 For a small $C \subseteq \Mon$, we say that a $\KK$-configuration $(I, f_A)_{A \in \KK}$ is \emph{over} $C$ if the image of $I$ is contained in the set of all $L(C)$-formulas.  We say $(I, f_A)_{A \in \KK}$ is \emph{parameter-free} if it is over $\emptyset$.  We say $(I, f_A)_{A \in \KK}$ is \emph{injective} if $f_A$ is injective for each $A \in \KK$.
\end{defn}

A $\KK$-configuration can be defined in terms of the \fraisse limit of $\KK$.

\begin{lem}\label{Lem_FiniteConfig}
 Let $\KK$ be \analgtriv \fraisse class in a finite relational language $L_0$ with \fraisse limit $\Gamma$, let $\pi$ be a partial type in $T$, and let $C \subseteq \Mon$ be small.  There exists a $\KK$-configuration into $\pi$ over $C$ if and only if there exist $I : \sig(R_0) \rightarrow L(C)$ and $f : \Gamma \rightarrow \pi(\Mon)$ such that, for all $R \in \sig(L_0)$ and for all $\oa \in \Gamma^{\arity(R)}$,
 \[
  \Gamma \models R(\oa) \mathiff \Mon \models I(R)(f(\oa)).
 \]
\end{lem}

\begin{proof}
 ($\Leftarrow$): Suppose $I$ and $f$ are given.  For each $A \in \KK$, let $f_A$ be obtained by composing $f$ with an embedding of $A$ into $\Gamma$.  Then, $(I, f_A)_{A \in \KK}$ is a $\KK$-configuration into $\pi$.
 
 ($\Rightarrow$): Let $(I, f_A)_{A \in \KK}$ be a $\KK$-configuration into $\pi(\oy)$.  For each $a \in \Gamma$, let $\oy_a$ be a tuple of variables in $T$ of the same sort as $\oy$ such that $\oy_a$ and $\oy_b$ are disjoint for $a \neq b$.  Consider the type $\Sigma$ in the free variables $(\oy_a)_{a \in \Gamma}$ consisting of:
 \begin{enumerate}
  \item $\pi(\oy_a)$ for all $a \in \Gamma$; and
  \item $I(R)(\oy_{a_0}, \dots, \oy_{a_{n-1}})^{\text{iff } \Gamma \models R(\oa)}$ for all $n$-ary $R \in \sig(L_0)$ and $\oa \in \Gamma^n$.
 \end{enumerate}
  By assumption, $\Sigma$ is finitely satisfiable.  By compactness, $\Sigma$ is consistent and, by saturation of $\Mon$, it has a realization in $\Mon$, say $( \oc_a )_{a \in \Gamma}$.  Define $f : \Gamma \rightarrow \pi(\Mon)$ by setting $f(a) = \oc_a$.  Then $I$ and $f$ are the desired functions.
 \if\papermode0
 { \color{violet}
 
  \emph{More Detail:}
  
  ($\Rightarrow$): Suppose that $f : \Gamma \rightarrow \pi(\Mon)$ is a $\KK$-configuration over $C$.  For each $R \in \sig(L_0)$, let $\varphi_R \in L(C)$ witness this (i.e., for all $\ob \in \Gamma^{\arity(R)}$, $\Gamma \models R(\ob)$ if and only if $\Mon \models \varphi_R(f(\ob))$).  Set $I(R) = \varphi_R$.  Then, for any $A \in \KK$, choose some embedding $g : A \rightarrow \Gamma$.  Thus, for all $R \in \sig(L_0)$, for all $\oa \in A^{\arity(R)}$, $A \models R(\oa)$ if and only if $\Gamma \models R(g(\oa))$ if and only if $\Mon \models I(R)(f(g(\oa)))$.
 
  ($\Leftarrow$): Consider the type
  \begin{align*}
   \Sigma(\oy_b)_{b \in \Gamma} = & \bigcup \{ \pi(\oy_b) : b \in \Gamma \} \cup \\ & \left\{ I(R)(\oy_{\ob})^{\text{iff } \Gamma \models R(\ob)} : R \in \sig(L_0), \ob \in \Gamma^{\arity(R)} \right\}.
  \end{align*}
  Choose some finite $\Sigma_0 \subseteq \Sigma$.  Let $A \subseteq \Gamma$ be finite such that, for each variable $\oy_b$ in a formula in $\Sigma_0$, $b \in A$.  By hypothesis on $A$, $\Sigma_0$ is consistent, witnessed by $(f(a))_{a \in A}$ for the given $f : A \rightarrow \pi(\Mon)$.  By compactness, $\Sigma$ is consistent.  Let $(\oc_b)_{b \in \Gamma}$ witness this.  Define $f : \Gamma \rightarrow \pi(\Mon)$ by setting $f(b) = \oc_b$ for all $b \in \Gamma$.  Then, for all $R \in \sig(L_0)$, for all $\oa \in \Gamma^{\arity(R)}$, $\Gamma \models R(\oa)$ if and only if $\Mon \models I(R)(f(\oa))$.  Thus, $f$ is a $\KK$-configuration into $\pi$ over $C$.
  }
 \fi
\end{proof}

In light of the previous lemma, configurations are closely related to a notion called ``trace definability,'' studied in \cite{wals}.

\begin{expl}\label{Expl_Config}
 If $T$ is the theory of the random $3$-hypergraph and $\pi(x) = (x = x)$, where $x$ is a singleton, then there exists a $\GG$-configuration into $\pi$.  Specifically, if $R$ is the ternary relation in $L$ and $E$ is the binary relation in $L_0$, then let $I(E)(x,y) = R(x,y,a)$ for some $a \in \Mon$; this gives the desired $\GG$-configuration.  However, there does not exist a parameter-free $\GG$-configuration into $\pi$; by quantifier elimination, any $L$-formula must be a Boolean combination of formulas of the form $R(x,y,z)$ and equality.
\end{expl}
 
We will see that any configuration can be made to be parameter-free at the cost of changing the type; see Lemma \ref{Lem_ParameterFreeConfig}.  On the other hand, if $T$ has NIP, we will see that there exists no $\GG$-configuration into any partial type in $T$ (see Theorem \ref{Thm_CCExamples} (2)).

In the preceding example, we saw a case where the target theory was the theory of a \fraisse limit of a \fraisse class other than the index class $\KK$.  In this case, we were unable to find a parameter-free configuration into $x = x$.  However, if $T$ is the theory of the \fraisse limit of $\KK$, we can.

\begin{lem}\label{Lem_IdentityConfig}
 Let $\KK$ be \analgtriv \fraisse class in a finite relational language $L_0$, let $T$ be the theory of the \fraisse limit of $\KK$, and let $\pi(x) = (x = x)$, where $x$ is a singleton.  Then, there exists a parameter-free $\KK$-configuration into $\pi$.
\end{lem}

\begin{proof}
 Let $I$ be the inclusion function on $\sig(L_0)$ and, for each $A \in \KK$, let $f_A : A \rightarrow \Mon$ be any embedding.  Then, $(I, f_A)_{A \in \KK}$ is a $\KK$-configuration into $\pi$.
\end{proof}

\begin{lem}\label{Lem_SmallerTypes}
 Let $\KK$ be \analgtriv \fraisse class in a finite relational language $L_0$.  If $\pi_0(\oy)$ and $\pi_1(\oy)$ are partial types in $T$, $\pi_0(\oy) \vdash \pi_1(\oy)$, and there exists a $\KK$-configuration in $\pi_0$, then there exists a $\KK$-configuration into $\pi_1$.
\end{lem}

\begin{proof}
 Since $\pi_0(\Mon) \subseteq \pi_1(\Mon)$, any $\KK$-configuration into $\pi_0$ is also a $\KK$-configuration into $\pi_1$.
\end{proof}

As an immediate consequence of the previous lemma, if $\pi(\oy)$ is a partial type in $T$ and there exists a $\KK$-configuration into $\pi$, then there exists a $\KK$-configuration into the type $\oy = \oy$.

As previously mentioned, we can convert any configuration into a parameter-free one at the cost of changing the target partial type.

\begin{lem}\label{Lem_ParameterFreeConfig}
 Let $\KK$ be \analgtriv \fraisse class in a finite relational language $L_0$ and let $\pi$ be a partial type in $T$.  If there exists a $\KK$-configuration into $\pi$, then there exists a parameter-free $\KK$-configuration into some partial type of $T$ (possibly different from $\pi$).
\end{lem}

\begin{proof}
 Let $\pi(\oy)$ be a partial type in $T$ and let $(I, f_A)_{A \in \KK}$ be a $\KK$-configuration into $\pi$.  Choose $\oc \in \Mon^{< \omega}$ such that, for all $R \in \sig(L_0)$, we can take $\oc$ to be the parameters of $I(R)$ (this can be done as $\sig(L_0)$ is finite).  Define $\pi^*$ a partial type of $T$ to be $\pi$ expanded by adding dummy variables $\oz$ for $\oc$.  Then, for each $R \in \sig(L_0)$ of arity $n$, $I(R)(\oy_0, \dots, \oy_{n-1})$ is $T$-equivalent to $\varphi_R(\oy_0, \dots, \oy_{n-1}, \oc)$ for some $L$-formula $\varphi_R$.  Let
 \[
  I'(R)(\oy_0, \oz_0, \oy_1, \oz_1, \dots, \oy_{n-1}, \oz_{n-1}) = \varphi_R(\oy_0, \dots, \oy_{n-1}, \oz_0).
 \]
 For each $A \in \KK$, define $f'_A : A \rightarrow \pi^*(\Mon)$ as follows: For $a \in A$,
 \[
  f'_A(a) = (f_A(a), \oc).
 \]
 Then, it is easy to check that $(I', f'_A)_{A \in \KK}$ is a parameter-free $\KK$-configuration into $\pi^*$.
\end{proof}

We can also convert any configuration into an injective one at the cost of changing the target partial type (assuming the target theory has infinite models).

\begin{lem}\label{Lem_InjectiveConfig}
 Assume $T$ has infinite models.  Let $\KK$ be \analgtriv \fraisse class in a finite relational language $L_0$ and let $\pi$ be a partial type in $T$.  If there exists a $\KK$-configuration into some partial type $\pi$ of $T$, then there exists an injective $\KK$-configuration into some partial type of $T$ (possibly different from $\pi$).
\end{lem}

\begin{proof}
 Let $\pi(\oy)$ be a partial type in $T$ and let $(I, f_A)_{A \in \KK}$ be a $\KK$-configuration into $\pi$.  Let $z$ be a single variable in $T$ not used in $\pi$ and let $\pi'(\oy, z) = \pi(\oy)$.  For each $n$-ary $R \in \sig(L_0)$, let
 \[
  I'(R)(\oy_0, z_0, \dots, \oy_{n-1}, z_{n-1}) = I(R)(\oy_0, \dots, \oy_{n-1}).
 \]
 For each $A \in \KK$, since $A$ is finite and $\Mon$ is infinite, there exists an injective function $g : A \rightarrow \Mon$.  Let $f'_A : A \rightarrow \pi'(\Mon)$ be given by $f'_A(a) = (f_A(a), g(a))$ for all $a \in A$.  By construction, $f'_A$ is injective.  Thus, $(I', f'_A)_{A \in \KK}$ is an injective $\KK$-configuration into $\pi'$.
\end{proof}

\begin{defn}\label{Defn_ReductiveSubclass}
 For each $t < 2$, let $\KK_t$ be \analgtriv \fraisse class over a finite relational language $L_t$.  We say that $\KK_0$ is a \emph{reductive subclass} of $\KK_1$ if $\sig(L_0) \subseteq \sig(L_1)$ and, for each $A \in \KK_0$, there exists $B \in \KK_1$ such that $A \cong_{L_0} B |_{L_0}$.
\end{defn}

\begin{expl}
 Note that $\LO$ is a reductive subclass of $\TT$ (it is actually just a subclass).  For any $\KK_0$ and $\KK_1$, $\KK_0$ is a reductive subclass of $\KK_0 \sstar \KK_1$ (see Remark \ref{Rem_Gluing}).
\end{expl}

\begin{lem}\label{Lem_RedSubClass}
 For each $t < 2$, let $\KK_t$ be \analgtriv \fraisse class over a finite relational language $L_t$, and let $\pi$ be a partial type in $T$.  If there exists a $\KK_1$-configuration into $\pi$ and $\KK_0$ is a reductive subclass of $\KK_1$, then there exists a $\KK_0$-configuration into $\pi$.
\end{lem}

\begin{proof}
 Fix $(I, f_B)_{B \in \KK_1}$ a $\KK_1$-configuration into $\pi$.  Fix $A \in \KK_0$.  Choose $B \in \KK_1$ and $g : A \rightarrow B$ such that $g$ is an $L_0$-isomorphism, and let $f'_A = f_B \circ g$.  Then, for all $R \in \sig(L_0)$ and $\oa \in A^{\arity(R)}$,
 \[
  A \models R(\oa) \mathiff B \models R(g(\oa)) \mathiff \Mon \models I(R)(f_B(g(\oa))).
 \]
 Thus, $(I, f'_A)_{A \in \KK_0}$ is a $\KK_0$-configuration into $\pi$.
\end{proof}

If $\pi_0(\oy_0)$ and $\pi_1(\oy_1)$ are two partial types in $T$ where $\oy_0$ and $\oy_1$ are disjoint, define $\pi_0 \times \pi_1$ to be the following type:
\[
 (\pi_0 \times \pi_1)(\oy_0, \oy_1) = \pi_0(\oy_0) \cup \pi_1(\oy_1).
\]
If $\oy_0$ and $\oy_1$ are not disjoint, we can choose different variables to force disjointness.  Fix $n \ge 1$ and define $\pi^{\times n}$ recursively as follows:
\begin{enumerate}
 \item $\pi^{\times 1} = \pi$,
 \item $\pi^{\times (n+1)} = \pi^{\times n} \times \pi$.
\end{enumerate}

It turns out that free superposition interacts with configurations into these type products in the obvious manner.

\begin{prop}\label{Prop_ProductTypes}
 For each $t < 2$, let $\KK_t$ be \analgtriv \fraisse class over a finite relational language $L_t$.  Suppose $\pi_0$ and $\pi_1$ are two partial types in $T$.  Suppose there exist a $\KK_0$-configuration into $\pi_0$ and a $\KK_1$-configuration into $\pi_1$.  Then, there exists a $(\KK_0 \sstar \KK_1)$-configuration into $\pi_0 \times \pi_1$.
\end{prop}

\begin{proof}
 For each $t < 2$, let $(I_t, f_{t,A})_{A \in \KK_t}$ be a $\KK_t$-configuration into $\pi_t$.  We build $(I, f_A)_{A \in \KK_0 \sstar \KK_1}$ a $(\KK_0 \sstar \KK_1)$-configuration into $\pi_0 \times \pi_1$.
 
 For each $t < 2$ and each $n$-ary relation symbol $R$ in $L_t$, let
 \[
  I(R)(\oy_{0,0}, \oy_{0,1}, \oy_{1,0}, \oy_{1,1}, \dots, \oy_{n-1,0}, \oy_{n-1,1}) = I_t(R)(\oy_{0,t}, \oy_{1,t}, \dots, \oy_{n-1,t}).
 \]
 (Note that $\oy_{i,t}$ is of the same sort as free the variables of $\pi_t$ for all $i < n$ and $t < 2$.)  Fix $A \in \KK_0 \sstar \KK_1$.  Let $f_A : A \rightarrow (\pi_0 \times \pi_1)(\Mon)$ be given by, for all $a \in A$,
 \[
  f_A(a) = (f_{0,A |_{L_0}}(a), f_{1, A|_{L_1}}(a)).
 \] 
 Then, we get that, for all $R \in \sig(L_2)$, for all $\oa \in A^{\arity(R)}$,
 \[
  A \models R(\oa) \mathiff \Mon \models I(R)(f_A(\oa)).
 \]
 Thus, $(I, f_A)_{A \in \KK_0 \sstar \KK_1}$ is a $(\KK_0 \sstar \KK_1)$-configuration into $\pi_0 \times \pi_1$.
\end{proof}

\begin{cor}\label{Cor_nIdentityConfig}
 Let $\KK$ be \analgtriv \fraisse class in a finite relational language and let $T_0$ be the theory of the \fraisse limit of $\KK$.  If $\ox$ is a tuple of variables with $n = |\ox|$ in $T_0$, then there exists a $\KK^{\sstar n}$-configuration into $\ox = \ox$.
\end{cor}

\begin{proof}
 Use Lemma \ref{Lem_IdentityConfig}, Proposition \ref{Prop_ProductTypes}, and induction.
\end{proof}

We can ``compose'' configurations, as long as the first configuration is parameter-free and the second is injective.

\begin{prop}\label{Prop_ComposeConfig}
 For each $t < 2$, let $\KK_t$ be \analgtriv \fraisse class over a finite relational language $L_t$.  Suppose that $\pi(\oz)$ is a partial type in $T$ and suppose $\oy$ is an $n$-tuple of variables in the \fraisse limit of $\KK_1$ for some $n < \omega$.  Suppose there exist an injective $\KK_1$-configuration into $\pi$ and a parameter-free $\KK_0$-configuration into $\oy = \oy$.  Then, there exists a $\KK_0$-configuration into $\pi^{\times n}$.
\end{prop}

\begin{proof}
 Let $(I, f_A)_{A \in \KK_0}$ be a $\KK_0$-configuration into $\oy = \oy$.  Since the theory of the \fraisse limit of $\KK_1$ has quantifier elimination, we may assume that, for each $R \in \sig(L_0)$, $I(R)$ is a quantifier-free $L_1$-formula.  Let $(J, g_B)_{B \in \KK_1}$ be a $\KK_1$-configuration into $\pi(\oz)$.  Extend $J$ to $\sig(L_1) \cup \{ = \}$ by setting
 \[
  J(=)(\oz_0, \oz_1) = [\oz_0 = \oz_1]. 
 \]
 Define $H : \sig(L_0) \rightarrow L(\Mon)$ by the following method:  For each $k$-ary $R \in \sig(L_0)$, consider $I(R)(\oy_0, \dots, \oy_{k-1})$ (so $\oy_i = ( y_{i,0}, \dots, y_{i,n-1} )$ for each $i < k$).  For each $S(y_{i_0,j_0}, \dots, y_{i_{\ell-1},j_{\ell-1}}) \in \sig(L_1) \cup \{ = \}$ used in $I(R)$, replace it with
 \[
  J(S)(\oz_{i_0,j_0}, \dots, \oz_{i_{\ell-1},j_{\ell-1}}).
 \]
 This creates an $L(\Mon)$-formula in the variables $( ( \oz_{i,j} )_{j < n} )_{i < k}$; call it $H(R)$.  For each $A \in \KK_0$, choose $B \in \KK_1$ so that the image of $f_A$ is contained in $B^n$ (we can do this since $f_A$ maps into $n$-tuples of the \fraisse limit of $\KK_1$).  Then, define $h_A : A \rightarrow \pi^{\times n}(\Mon)$ by setting $h_A(a) = g_B(f_A(a))$ for all $a \in A$.  It is easy to check that $(H, h_A)_{A \in \KK_0}$ is a $\KK_0$-configuration into $\pi^{\times n}$.
 \if\papermode0
 { \color{violet}
 
  \emph{More Detail:} Fix $A \in \KK_0$.  By Lemma \ref{Lem_FiniteConfig}, there exists $f : A \rightarrow \Gamma_1^n$ such that, for all $R \in \sig(L_0)$, for all $\oa \in A^{\arity(R)}$,
 \[
  A \models R(\oa) \mathiff \Gamma_1 \models I(R)(f(\oa)).
 \]
 Let $B = \{ f(a)_i : a \in A, i < n \} \in \KK_1$.  By Lemma \ref{Lem_FiniteConfig}, there exists $g : B \rightarrow \pi(\Mon)$ such that, for all $S \in \sig(L_1)$, for all $\ob \in B^{\arity(R)}$,
 \[
  B \models S(\ob) \mathiff \Mon \models J(S)(g(\ob)).
 \]
 Therefore, for all $R \in \sig(L_0)$, for all $\oa \in A^{\arity(R)}$,
 \[
  A \models R(\oa) \mathiff B \models I(R)(f(\oa)) \mathiff \Mon \models H(R)(g(f(\oa))).
 \]
 That is, $H$ satisfies the conditions of Lemma \ref{Lem_FiniteConfig}.

 }
 \fi
\end{proof}

We analyze how the properties of being \selfsim\ and being age indivisible translate to configurations.  Being age indivisible manifests in a uniformity condition on $L$-types.

\begin{prop}\label{Prop_WeakGoodEmbedding}
 Suppose that $\KK$ is an age indivisible \algtriv \fraisse class in a finite relational language $L_0$ with \fraisse limit $\Gamma$ and suppose that $\pi(\oy)$ is a partial type in $T$.  For all small $C \subseteq \Mon$, if there exists a $\KK$-configuration into $\pi$ over $C$, then there exist $I : \sig(R_0) \rightarrow L(C)$ and $f : \Gamma \rightarrow \pi(\Mon)$ such that
 \begin{enumerate}
  \item for all $R \in \sig(L_0)$, for all $\oa \in \Gamma^{\arity(R)}$, $\Gamma \models R(\oa)$ if and only if $\Mon \models I(R)(f(\oa))$; and
  \item for all $a, b \in \Gamma$,
   \[
    \tp_L(f(a) / C) = \tp_L(f(b) / C).
   \]
 \end{enumerate}
\end{prop}

\begin{proof}
 Take $I : \sig(L_0) \rightarrow L(C)$ and $f : \Gamma \rightarrow \pi(\Mon)$ as in Lemma \ref{Lem_FiniteConfig}.  Take $\Sigma$ as in the proof of Lemma \ref{Lem_FiniteConfig}, with the additional formulas:
 \begin{enumerate}
  \setcounter{enumi}{2}
  \item $\psi(\oy_a) \leftrightarrow \psi(\oy_b)$ for all $\psi \in L(C)$ and $a, b \in \Gamma$.
 \end{enumerate}
 Fix $\Sigma_0 \subseteq \Sigma$ finite.  Then, there exists a finite set of $L(C)$-formulas $\Psi(\oy)$ and a finite $A \subseteq \Gamma$ so that $\Sigma_0$ mentions only variables $\oy_a$ for $a \in A$ and only formulas $\psi(\oy_a) \leftrightarrow \psi(\oy_b)$ for $\psi \in \Psi$ and $a, b \in A$.  Consider the coloring $c : \Gamma \rightarrow {}^{\Psi} 2$ so that, for all $a \in \Gamma$ and $\psi \in \Psi$,
 \[
  c(a)(\psi) = 1 \mathiff \Mon \models \psi(f(a)).
 \]
 Since $\KK$ is age indivisible, there exists an embedding $g : A \rightarrow \Gamma$ such that $c$ is constant on $g(A)$.  Thus, we get. $(f(g(a)))_{a \in A} \models \Sigma_0$.
 
 By compactness and saturation, there exists $(\oc_a)_{a \in \Gamma} \models \Sigma$.  Define $f' : \Gamma \rightarrow \pi(\Mon)$ by setting $f'(a) = \oc_a$.  Then, $f'$ is the desired function.
\end{proof}

Note that, if we take $C$ so that $\pi$ is over $C$, then Proposition \ref{Prop_WeakGoodEmbedding} is saying that we can choose $f$ so that $f$ maps into the realizations of some complete type over $C$ extending $\pi$.

When $\KK$ is \selfsim, we get a stronger condition on $L$-types.  For all $A \subseteq \Gamma$, let $\cS(A)$ be the set of all complete, non-algebraic $1$-$L_0$-types over $A$.

\begin{prop}\label{Prop_GoodEmbedding}
 Suppose that $\KK$ is a \selfsim\ \algtriv \fraisse class in a finite relational language $L_0$ with \fraisse limit $\Gamma$ and suppose that $\pi(\oy)$ is a partial type in $T$.  For all small $C \subseteq \Mon$, if there exists a $\KK$-configuration into $\pi$ over $C$, then there exist $I : \sig(R_0) \rightarrow L(C)$, $f : \Gamma \rightarrow \pi(\Mon)$, and $J \subseteq |\oy|$ such that
 \begin{enumerate}
  \item for all $R \in \sig(L_0)$, for all $\oa \in \Gamma^{\arity(R)}$, $\Gamma \models R(\oa)$ if and only if $\Mon \models I(R)(f(\oa))$; and
  \item for all $a, b \in \Gamma$, $\tp_L(f(a) / C) = \tp_L(f(b) / C)$;
  \item for all $j \in J$ and all $a, b \in \Gamma$, $f(a)_j = f(b)_j$; and
  \item for all finite $A \subseteq \Gamma$ and all $p \in \cS(A)$, there exists $b \models p$ such that, for all $i, j \in |\oy| \setminus J$ and all $a \in A$, $f(a)_i \neq f(b)_j$.
 \end{enumerate}
\end{prop}

\begin{proof}
 By Proposition \ref{Prop_WeakGoodEmbedding}, there exist $I$ and $f$ such that (1) and (2) hold.
 
 For conditions (3) and (4), start with $J = \emptyset$ and construct $J$ recursively as follows:  For any $J$ satisfying condition (3), assume that condition (4) fails.  So there exist a finite $A \subseteq \Gamma$ and $p \in \cS(A)$ such that, for all $b \models p$, there exist $i, j \in |\oy| \setminus J$ and $a \in A$ such that $f(a)_i = f(b)_j$.  Let $\Gamma' = \{ b \in \Gamma : b \models p \}$.  Since $\KK$ is 
\selfsim, $\Gamma' \cong \Gamma$.  Consider a coloring $c : \Gamma' \rightarrow (|\oy| \setminus J)^2 \times A$ given by $c(b) = (i,j,a)$ for some choice of $i, j \in |\oy| \setminus J$ and $a \in A$ such that $f(a)_i = f(b)_j$.  By Lemma \ref{Lem_Coloring}, we may assume that $c$ is constant.  Thus, for all $b, d \in \Gamma'$, $f(b)_j = f(a)_i = f(d)_j$ (in other words, condition (3) holds on $\Gamma'$ for $J \cup \{ j \}$).  Add $j$ to $J$ and replace $\Gamma$ with $\Gamma'$.  Repeat this process.  Since $|\oy|$ is finite, this will eventually terminate.  This gives us the desired conclusion.
\end{proof}

We use Proposition \ref{Prop_WeakGoodEmbedding} and Proposition \ref{Prop_GoodEmbedding} in Subsection \ref{Subsect_RandGraph} to compute $\KK$-ranks for particular choices of $\KK$.

%%%%%%%%%%%%%%%%%%%%%%%%%%%%%%%
%% Section -- Dividing Lines %%
%%%%%%%%%%%%%%%%%%%%%%%%%%%%%%%

\section{Dividing Lines}\label{Sect_DividingLines}

Before defining and studying $\KK$-ranks, we first connect the notions discussed above with the ideas considered in \cite{gh}.

\begin{defn}
 Let $\KK$ be \analgtriv \fraisse class.  Define $\cC_\KK$ to be the class of all complete theories $T$ with infinite models such that there exists a $\KK$-configuration into some partial type $\pi$ (in this case, we will say that $T$ \emph{admits} a $\KK$-configuration).
\end{defn}

Note that our definition of $\cC_\KK$ coincides with the definition of $\mathfrak{C}_\KK$ from \cite{gh} in the case where $\KK$ is an indecomposable \algtriv \fraisse class (see Observation 2.12 of \cite{gh}).  In that paper, the authors establish a quasi-order on theories, use this quasi-order to define classes of theories, and show that these classes are exactly those of the form $\cC_\KK$ for some indecomposable \algtriv \fraisse class, $\KK$ (see Theorem 2.17 of \cite{gh} for more details).

How do the classes $\cC_\KK$ relate to known dividing lines in model theory?  First of all, $\cC_\BS = \cC_\EE$ is the class of all complete theories with infinite models.  What about more interesting $\KK$?  The following theorem describes the relationship of $\cC_\KK$ to the classes of theories that are stable, NIP, and $k$-dependent.

\begin{thm}[Proposition 4.31 of \cite{gh}, Proposition 5.2 of \cite{cpt}]\label{Thm_CCExamples}
 Let $T$ be a complete first-order theory with infinite models.
 \begin{enumerate}
  \item $T$ is stable if and only if $T \notin \cC_\LO$.
  \item $T$ has NIP if and only if $T \notin \cC_\GG$.
  \item For all $k \ge 2$, $T$ has $(k-1)$-dependence if and only if $T \notin \cC_{\HH_k}$.
 \end{enumerate}
\end{thm}

\begin{proof}
 If $\varphi(\oy; \oz)$ is a witness to the order property, then the map $I$ sending $<$ to
 \[
  \varphi^*(\oy_0, \oz_0; \oy_1, \oz_1) = \varphi(\oy_0; \oz_1)
 \]
 witnesses that there exists an $\LO$-configuration into $\Mon^{|\oy|+|\oz|}$.  Similar arguments can be made for the $(k-1)$-independence property and $\HH_k$-configurations for $k \ge 2$; see the proof of Lemma 2.2 of \cite{ls03}.
 \if\papermode0
 { \color{violet}
 
  \emph{More Detail:} 
  
  (1), ($\Rightarrow$): Assume that $T \in \cC_\LO$.  Then, there exists an $\LO$-configuration $f : \Gamma \rightarrow \Mon^n$ for some $n < \omega$.  In other words, there exists an $L(\Mon)$-formula $\varphi(\oy_0, \oy_1)$ such that, for all $a, b \in \Gamma$,
  \[
   \Gamma \models a < b \mathiff \Mon \models \varphi(f(a), f(b)).
  \]
  That is, $\varphi(\oy_0, \oy_1)$ has the order property.
 
  (1), ($\Leftarrow$): Assume $\varphi(\oy; \oz)$ has the order property.  Define $I : \sig(L_0) \rightarrow L(\Mon)$ by setting $I(<)(\oy_0, \oz_0; \oy_1, \oz_1) = \varphi(\oy_0; \oz_1)$.  For any $A \in \Gamma$, choose $\oc_a \in \Mon^{|\oy|}$ and $\od_a \in \Mon^{|\oz|}$ for all $a \in A$ such that, for all $a, b \in \Gamma$,
  \[
   \Gamma \models a < b \mathiff \Mon \models \varphi(\oc_a; \od_b).
  \]
  Define $f : A \rightarrow \Mon^{|\oy|+|\oz|}$ by $f(a) = (\oc_a, \od_a)$.  Then, for all $a, b \in A$,
  \[
   \Gamma \models a < b \mathiff \Mon \models I(<)(f(a), f(b)).
  \]
  By Lemma \ref{Lem_FiniteConfig}, there exists an $\LO$-configuration.
 
  (2) follows from (3) when $k = 2$.
 
  (3), ($\Rightarrow$): Assume that $T \in \cC_{\HH_k}$.  Then, there exists an $\HH_k$-configuration $f : \Gamma \rightarrow \Mon^n$ for some $n < \omega$.  In other words, there exists an $L(\Mon)$-formula $\varphi(\oy_0, \dots \oy_{k-1})$ such that, for all $\oa \in \Gamma^k$,
  \[
   \Gamma \models E(\oa) \mathiff \Mon \models \varphi(f(\oa)).
  \]
  Fix $m < \omega$ and choose distinct $a_{i,j} \in \Gamma$ for $i < k - 1$ and $j < m$.  Since $\Gamma$ is the random $k$-hypergraph, for each $s \subseteq m^{k-1}$, there exists $b_s \in \Gamma$ such that, for all $\oj \in m^{k-1}$,
  \[
   \Gamma \models E(a_{0,j_0}, \dots, a_{k-2,j_{k-2}}, b_s) \mathiff \oj \in s.
  \]
  Therefore, for all $\oj \in m^{k-1}$ and $s \subseteq m^{k-1}$,
  \[
   \Mon \models \varphi(f(a_{0,j_0}), \dots, f(a_{k-2,j_{k-2}}), f(b_s)) \mathiff \oj \in s.
  \]
  Since $m$ was arbitrary, by compactness, $\varphi$ has the $(k-1)$-independence property.  Thus, $T$ has the $(k-1)$-independence property.
  
  (3), ($\Leftarrow$): Assume that $T$ has $(k-1)$-independence.  Thus, there exist an $L(\Mon)$-formula $\varphi(\oz_0, \oz_1, \dots, \oz_{k-1})$, $\oa_{i,j} \in \Mon^{|\oz_i|}$ for $i < k - 1$ and $j < \omega$, and $\ob_s \in \Mon^{|\oz_{k-1}|}$ for $s \subseteq \omega^{k-1}$ such that, for all $\oj \in \omega^{k-1}$,
  \[
   \Mon \models \varphi(\oa_{0,j_0}, \dots, \oa_{k-2, j_{k-2}}, \ob_s) \mathiff \oj \in s.
  \]
  Define
  \[
   \varphi'(\oz_{i,j})_{i,j<k} = \varphi(\oz_{0,0}, \dots, \oz_{k-1,k-1}).
  \]
  Let $I$ be the function that sends $E$ to $\varphi'$.  Fix $G$ a $k$-hypergraph with vertex set $m < \omega$.  For each $v \in G$, let
  \[
   S_v = \{ \overline{w} \in G^{k-1} : G \models E(w_0, \dots, w_{k-2}, v) \}.
  \]
  Finally, define $f : G \rightarrow \Mon^{|\oz_0| + \dots + |\oz_{k-1}|}$ as follows: For all $v \in G$, let
  \[
   f(v) = ( \oa_{0,v}, \dots, \oa_{k-2,v}, \ob_{S_v} ).
  \]
  Then, for all $\overline{v} \in G^k$, $G \models E(\overline{v})$ if and only if $(v_0, \dots, v_{k-2}) \in S_{v_k}$ if and only if $\Mon \models \varphi(\oa_{0,v_0}, \dots, \oa_{k-2,v_{k-2}}, \ob_{S_{v_{k-1}}})$ if and only if $\Mon \models \varphi'(f(\overline{v}))$.  Thus, by Lemma \ref{Lem_FiniteConfig}, $T$ admits a $\HH_k$-configuration.
 }
 \fi 
\end{proof}

\begin{defn}
 Given two \algtriv \fraisse classes $\KK_0$ and $\KK_1$, we say that
 \[
  \KK_0 \lesssim \KK_1
 \]
 if the theory of the \fraisse limit of $\KK_1$ is in $\cC_{\KK_0}$.  We say
 \[
  \KK_0 \sim \KK_1
 \]
 if $\KK_0 \lesssim \KK_1$ and $\KK_1 \lesssim \KK_0$.
\end{defn}

\begin{prop}\label{Prop_AllAboutLesssim}
 Fix \algtriv \fraisse classes $\KK_0$, $\KK_1$, and $\KK_2$.
 \begin{enumerate}
  \item $\lesssim$ is a quasi-order on \algtriv \fraisse classes.
  \item $\KK_0 \lesssim \KK_0 \sstar \KK_1$.
  \item if $\KK_0 \lesssim \KK_2$ and $\KK_1 \lesssim \KK_2$, then $\KK_0 \sstar \KK_1 \lesssim \KK_2$.
  \item $\KK_0 \lesssim \KK_1$ if and only if $\cC_{\KK_1} \subseteq \cC_{\KK_0}$.
  \item $\KK_0 \sim \KK_1$ if and only if $\cC_{\KK_0} = \cC_{\KK_1}$.
 \end{enumerate}
\end{prop}

\begin{proof}
 For each $i < 3$, let $L_i$ be the language of $\KK_i$ and let $T_i$ be the theory of the \fraisse limit of $\KK_i$.

 (1): By Lemma \ref{Lem_IdentityConfig}, $T_0$ admits a $\KK_0$-configuration.  Hence, $\KK_0 \lesssim \KK_0$.  So $\lesssim$ is reflexive.
 
 Assume that $\KK_0 \lesssim \KK_1$ and $\KK_1 \lesssim \KK_2$.  Then, $T_1$ admits a $\KK_0$-configuration and $T_2$ admits a $\KK_1$-configuration.  By Lemma \ref{Lem_ParameterFreeConfig}, $T_1$ admits a parameter-free $\KK_0$-configuration and, by Lemma \ref{Lem_InjectiveConfig}, $T_2$ admits an injective $\KK_1$-configuration.  So, by Proposition \ref{Prop_ComposeConfig}, $T_2$ admits a $\KK_0$-configuration.  Thus, $\KK_0 \lesssim \KK_2$.  So $\lesssim$ is transitive.
 
 (2):  Let $T_0 \sstar T_1$ be the theory of the \fraisse limit of $\KK_0 \sstar \KK_1$.  By (1), $T_0 \sstar T_1$ admits a $(\KK_0 \sstar \KK_1)$-configuration.  However, $\KK_0$ is a reductive subclass of $\KK_0 \sstar \KK_1$.  By Lemma \ref{Lem_RedSubClass}, $T_0 \sstar T_1$ admits a $\KK_0$-configuration.
 
 (3): Assume $\KK_0 \lesssim \KK_2$ and $\KK_1 \lesssim \KK_2$.  Thus, $T_2$ admits a $\KK_0$-configuration and a $\KK_1$-configuration.  By Proposition \ref{Prop_ProductTypes}, $T_2$ admits a $(\KK_0 \sstar \KK_1)$-configuration.  Therefore, $\KK_0 \sstar \KK_1 \lesssim \KK_2$.
 
 (4), ($\Rightarrow$): Assume $\KK_0 \lesssim \KK_1$ and $T \in \cC_{\KK_1}$.  By Lemma \ref{Lem_ParameterFreeConfig}, $T_1$ admits a parameter-free $\KK_0$-configuration.  By Lemma \ref{Lem_InjectiveConfig}, $T$ admits an injective $\KK_1$-configuration.  By Proposition \ref{Prop_ComposeConfig}, $T$ admits a $\KK_0$-configuration.  Thus, $T \in \cC_{\KK_0}$.
 
 (4), ($\Leftarrow$): Assume $\cC_{\KK_1} \subseteq \cC_{\KK_0}$.  By Lemma \ref{Lem_IdentityConfig}, $T_1$ is in $\cC_{\KK_1}$.  Therefore, it is in $\cC_{\KK_0}$.  Therefore, $\KK_0 \lesssim \KK_1$.
 
 (5): Follows immediately from (4).
\end{proof}

From this proposition, we get a characterization of when a free superposition of two classes is equivalent to one of the classes.  This corollary is a generalization of two results from \cite{gh}, namely Corollary 3.10 and Theorem 4.24.

\begin{cor}\label{Cor_KSim}
 Suppose $\KK_0$ and $\KK_1$ are \algtriv \fraisse classes.  Then, $\KK_0 \lesssim \KK_1$ if and only if $\KK_0 \sstar \KK_1 \sim \KK_1$.
\end{cor}

\begin{proof}
 ($\Rightarrow$): Assume $\KK_0 \lesssim \KK_1$.  By Proposition \ref{Prop_AllAboutLesssim} (2), $\KK_1 \lesssim \KK_0 \sstar \KK_1$.  By Proposition \ref{Prop_AllAboutLesssim} (1), $\KK_1 \lesssim \KK_1$.  By Proposition \ref{Prop_AllAboutLesssim} (3), $\KK_0 \sstar \KK_1 \lesssim \KK_1$.  Thus, $\KK_0 \sstar \KK_1 \sim \KK_1$.
 
 ($\Leftarrow$): Assume $\KK_0 \sstar \KK_1 \sim \KK_1$.  By Proposition \ref{Prop_AllAboutLesssim} (2), $\KK_0 \lesssim \KK_0 \sstar \KK_1$.  By Proposition \ref{Prop_AllAboutLesssim} (1), $\KK_0 \lesssim \KK_1$.
\end{proof}

\begin{cor}\label{Cor_StarSim}
 If $\KK$ is \analgtriv \fraisse class and $n \ge 1$, then $\KK^{\sstar n} \sim \KK$.
\end{cor}

\begin{proof}
 Follows from Corollary \ref{Cor_KSim} by induction.
\end{proof}

\begin{cor}(Corollary 3.10 of \cite{gh})
 Let $\KK$ be \analgtriv \fraisse class and $T$ the theory of the \fraisse limit of $\KK$.  Then, $T$ is unstable if and only if $\KK \sstar \LO \sim \KK$.
\end{cor}

\begin{proof}
 By Theorem \ref{Thm_CCExamples} (1), $T$ is unstable if and only if $T \in \cC_{\LO}$, which holds if and only if $\LO \lesssim \KK$.  By Corollary \ref{Cor_KSim}, $\LO \lesssim \KK$ if and only if $\KK \sstar \LO \sim \KK$.
\end{proof}

\begin{cor}(Theorem 4.24 of \cite{gh})
 Let $L_0$ be a finite relational language where each relation symbol is at least binary.  Let $\HH_{L_0}$ be the class of all $L_0$-hypergraphs (see Example \ref{Expl_GeneralHypergraphs}).  Let $k$ be the largest arity among relation symbols in $L_0$.  Then,
 \[
  \HH_{L_0} \sim \HH_k.
 \]
\end{cor}

\begin{proof}
 Let $k_0 \le \dots \le k_{n-1} = k$ list off all arities (with repetition) of the relation symbols in $L_0$.  Then,
 \[
  \HH_{L_0} = \HH_{k_0} \sstar \dots \sstar \HH_{k_{n-1}}.
 \]
 Notice that $\HH_m \lesssim \HH_\ell$ for each $m \le \ell$.  To see this, suppose that $E$ is the $m$-ary relation symbol for $\HH_m$, $R$ is the $\ell$-ary relation symbol for $\HH_\ell$, and $\Mon$ is a monster model of the theory of the \fraisse limit of $\HH_\ell$.  Fix $\oc \in \Mon^{\ell - m}$ and set
 \[
  I(E)(x_0, \dots, x_{m-1}) = R(x_0, \dots, x_{m-1}, \oc).
 \]
 It is easy to check that this creates an $\HH_m$-configuration into $\Mon$ (similar to Example \ref{Expl_Config}).
 
 So, by Corollary \ref{Cor_KSim},
 \[
  \HH_{k_0} \sstar \dots \sstar \HH_{k_{n-1}} \sim \HH_{k_{n-1}}.
 \]
 Therefore, $\HH_{L_0} \sim \HH_k$.
\end{proof}

%%%%%%%%%%%%%%%%%%%%%%%%%%%%
%% Section -- $\KK$-Ranks %%
%%%%%%%%%%%%%%%%%%%%%%%%%%%%

\section{$\KK$-Ranks}\label{Sect_Ranks}

In this section, instead of looking at which theories admit a $\KK$-configuration (into any type) for some \algtriv \fraisse class $\KK$, we want to pay close attention to a fixed partial type in the target.  We aim to count the number of ``independent copies'' of a single class $\KK$ that we can code into a partial type.  Let $T$ be a complete $L$-theory with monster model $\Mon$ and let $\KK$ be \analgtriv \fraisse class in a finite relational language $L_0$.

\begin{defn}\label{Defn_KKRank}
 Fix $n \ge 1$ and $\pi$ a partial type in $T$.  We say that $\pi$ has \emph{$\KK$-rank} $n$ if
 \begin{enumerate}
  \item there exists a $\KK^{\sstar n}$-configuration into $\pi$, and
  \item there does not exist a $\KK^{\sstar (n+1)}$-configuration into $\pi$.
 \end{enumerate}
 We say $\pi$ has $\KK$-rank $\infty$ if there exists a $\KK^{\sstar n}$-configuration into $\pi$ for all $n < \omega$.  We say $\pi$ has $\KK$-rank $0$ if there does not exist a $\KK$-configuration into $\pi$.  We denote the $\KK$-rank of $\pi$ by $\Rk_\KK(\pi)$.
\end{defn}

We can apply Lemma \ref{Lem_SmallerTypes} and Proposition \ref{Prop_ProductTypes} to get a few immediate results about $\KK$-rank.

\begin{prop}[Superadditivity of $\KK$-rank]\label{Prop_SuperAdd}
 For all partial types $\pi_0$ and $\pi_1$,
 \[
  \Rk_\KK(\pi_0 \times \pi_1) \ge \Rk_\KK(\pi_0) + \Rk_\KK(\pi_1).
 \]
\end{prop}

\begin{proof}
 Follows immediately from Proposition \ref{Prop_ProductTypes}.
\end{proof}

\begin{defn}\label{Defn_Additive}
 We say $\Rk_\KK$ is \emph{additive} if, for all partial types $\pi_0$ and $\pi_1$, if $\Rk_\KK(\pi_0) < \infty$ and $\Rk_\KK(\pi_1) < \infty$, then
 \[
  \Rk_\KK(\pi_0 \times \pi_1) = \Rk_\KK(\pi_0) + \Rk_\KK(\pi_1).
 \]
\end{defn}

For example, dp-rank is additive in the above sense \cite{kou}.  Similarly, op-dimension is additive \cite{gh2}.  This leads to the following question.

\begin{ques}\label{Ques_KKRankAdd}
 Under what conditions on $\KK$ and $T$ is $\KK$-rank additive?
\end{ques}

We present some partial results to Question \ref{Ques_KKRankAdd} later in this section (see Example \ref{Expl_LONIP} and Example \ref{Expl_GGNIP}).

\begin{lem}\label{Lem_RankMonotonicity}
 If $\pi_0(\oy)$ and $\pi_1(\oy)$ are partial types in $T$ and $\pi_0(\oy) \vdash \pi_1(\oy)$, then
 \[
  \Rk_\KK(\pi_0) \le \Rk_\KK(\pi_1).
 \]
\end{lem}

\begin{proof}
 Follows immediately from Lemma \ref{Lem_SmallerTypes}.
\end{proof}

Overloading notation, for each $n \ge 1$, we can define $\Rk_\KK(n)$ as follows: Fix an arbitrary $n$-tuple of variables $\oy$ from $T$ and set
\[
 \Rk_\KK(n) = \Rk_\KK(\oy = \oy).
\]
This is clearly independent of the choice of $\oy$.

\begin{lem}\label{Lem_RkKKOrder}
 For all $1 \le n < m < \omega$,
 \[
  \Rk_\KK(n) \le \Rk_\KK(m).
 \]
\end{lem}

\begin{proof}
 Suppose $\Rk_\KK(n) = \ell$.  Then, there exists a $\KK^{\sstar \ell}$-configuration into $\Mon^n$.  Clearly there exists an $\BS$-configuration into $\Mon^{m-n}$.  By Proposition \ref{Prop_ProductTypes}, there exists a $(\KK^{\sstar \ell} \sstar \BS)$-configuration into $\Mon^m$.  Since $\KK^{\sstar \ell} \sstar \BS = \KK^{\sstar \ell}$, we get that $\Rk_\KK(m) \ge \ell$.
\end{proof}

Using the terminology of Section \ref{Sect_DividingLines}, note that $T \in \cC_\KK$ if and only if $T$ is a complete theory with infinite models such that $\Rk_\KK(n) > 0$ for some $n \ge 1$.  Moreover, by Corollary \ref{Cor_StarSim}, if $T \in \cC_{\KK}$, then $T$ has types with arbitrarily large $\KK$-rank and, if $T \notin \cC_{\KK}$ (but $T$ is a complete theory with infinite models), then all types in $T$ have $\KK$-rank $0$.

In the following subsections, we will analyze $\KK$-rank for particular choices of $\KK$.

%%%%%%%%%%%%%%%%%%%%%%%%%%%%%%%%%%%%%
%% Subsection -- Linear Order Rank %%
%%%%%%%%%%%%%%%%%%%%%%%%%%%%%%%%%%%%%

\subsection{Linear Order Rank}\label{Subsect_LOR}

For this subsection, we consider the \algtriv \fraisse class $\LO$.  For any $m \ge 1$, let $L_m$ be the language of $\LO^{\sstar m}$, which consists of $m$ binary relation symbols $<_i$ for $i < m$.

It turns out that $\LO$-rank is closely related to op-dimension.  For any partial type in any theory, we get that the op-dimension is an upper bound for the $\LO$-rank.  Moreover, if the target theory has NIP, then op-dimension coincides with $\LO$-rank.

\begin{prop}\label{Prop_opDimUpperBound}
 For any partial type $\pi$ in any theory $T$,
 \[
  \Rk_\LO(\pi) \le \opDim(\pi).
 \]
\end{prop}

\begin{proof}
 Assume $\Rk_\LO(\pi) \ge m$.  Let $\Mon$ be a monster model for $T$ and let $(I, f_A)_{A \in \LO^{\sstar m}}$ be an $\LO^{\sstar m}$-configuration.  So, for all $A \in \LO^{\sstar m}$, for all $a, b \in A$, for all $i < m$,
 \[
  A \models a <_i b \mathiff \Mon \models I(<_i)(f_A(a), f_A(b)).
 \]
 Fix $n < \omega$.  We will use this configuration to build an IRD-pattern of depth $m$ and length $n$ in $\pi$.  Begin by creating an $A \in \LO^{\sstar m}$ with universe ${}^m (2n)$ by setting, for all $g, h \in A$ and $i < m$, $g <_i h$ if $g(i) < h(i)$ or $g(i) = h(i)$ and $g(j) < h(j)$ for minimal $j < m$ such that $g(j) \neq h(j)$.  Clearly, for all $i < m$, for all $g, h \in A$,
 \[
  g(i) < h(i) \Longrightarrow A \models g <_i h
 \]
 (but not conversely).
 
 For $g \in {}^m n$, let $\od_g = f_A(g')$, where $g' \in {}^m (2n)$ is such that $g'(i) = 2g(i) + 1$ for all $i < m$.  For $j < n$, let $\oc_j = f_A(h')$, where $h' \in {}^m (2n)$ is such that $h'(i) = 2j$ for all $i < m$.  Then, for all $g \in {}^m n$, $i < m$, and $j < n$,
 \[
  \Mon \models I(<_i)(\od_g, \oc_j) \mathiff 2g(i) + 1 < 2j \mathiff g(i) < j.
 \]
 Thus, for each $g \in {}^m n$,
 \[
  \pi(\oy) \cup \{ I(<_i)(\oy, \oc_j)^{\condiff g(i) < j} : i < m, j < n \}
 \]
 is consistent.  This is an IRD-pattern of depth $m$ and length $n$ in $\pi$.  Since $n$ was arbitrary, by compactness, $\pi$ has op-dimension $\ge m$.
\end{proof}

Before showing that $\LO$-rank and op-dimension coincide for NIP theories, we mention a trick (essentially Proposition 1.18 of \cite{gh2}) that follows from the fact that $\LO^{\sstar m}$ has the Ramsey property.

 %% Vince, 11/11/2022: Fixed this to the reviewer's specifications.
 %% Vince, 01/17/2023: Fixed this further.
\begin{rem}\label{Rem_GenIndisc}
 Suppose that $T$ is a complete $L$-theory for any language $L$, $\pi$ is a partial type in $T$, $\KK$ is a Ramsey class with \fraisse limit $\Gamma$, and there exists a $\KK$-configuration into $\pi$.  Then, by Theorem 3.12 of \cite{scow14}, the function $f$ in Lemma \ref{Lem_FiniteConfig} can be chosen so that, for all $n < \omega$, $\oa, \ob \in \Gamma^n$,
 \[
  \qftp_{L_0}(\oa) = \qftp_{L_0}(\ob) \Longrightarrow \tp_L(f(\oa)) = \tp_L(f(\ob)).
 \]
 By Corollary 1.4(2) of \cite{bod1}, $\LO^{\sstar m}$ is a Ramsey class; therefore, in particular, this holds when $\KK = \LO^{\sstar m}$.
\end{rem}

\begin{prop}\label{Prop_NIPLOopdim}
 If $T$ has NIP, then, for all partial types $\pi$,
 \[
  \Rk_\LO(\pi) = \opDim(\pi).
 \]
\end{prop}

This proof loosely follows the proof of Theorem 3.4 of \cite{ghs}, modified to fit into our current framework.

\begin{proof}
 The previous proposition gives us $\Rk_\LO(\pi) \le \opDim(\pi)$.  To prove the converse, suppose that $\pi$ has op-dimension $\ge m$.  Therefore, there exists an IRD-pattern of depth $m$ and length $\omega$ in $\pi$.  That is, there exist $L(\Mon)$-formulas $\varphi_i(\oy, \oz_i)$ for $i < m$ and $\oc_{i,j} \in \Mon^{|\oz_i|}$ for $i < m$ and $j < \omega$ such that, for all $g : m \rightarrow \omega$, the partial type
 \[ 
  \pi(\oy) \cup \{ \varphi_i(\oy, \oc_{i,j})^{\condiff g(i) < j} : i < m, j < \omega \}
 \]
 is consistent.  Say it is realized by $\ob_g \in \Mon^{|\oy|}$.  By coding tricks, we may assume that there exists an $L$-formula $\varphi(\oy, \oz)$ such that $\varphi_i = \varphi$ for all $i < m$.
 \if\papermode0
 { \color{violet}
 
  \textit{More Detail:} Let
  \[
   \varphi(\oy, \oz_0, \dots, \oz_{m-1}, w_0, \dots, w_{m-1}, u) = \bigwedge_{i < m} (w_i = u) \rightarrow \varphi_i(\oy, \oz_i).
  \]
  Choose $a \neq b$ in $\Mon$ and, for each $i < m$ and $j < \omega$, let
  \[
   \od_{i,j} = \oc_{i,j} \concat (a, \dots, a, b, a, \dots, a, b)
  \]
  (where the first $b$ appears in the $i$th entry).  Then, it is clear that
  \[
   \Mon \models \forall \oy \left( \varphi_i(\oy, \oc_{i,j}) \leftrightarrow \varphi(\oy, \od_{i,j}) \right).
  \]
 }
 \fi
 
 First, we create a function $f : \Gamma \rightarrow \pi(\Mon)$, where $\Gamma$ is the \fraisse limit of $\LO^{\sstar m}$.  Fix $A \in \LO^{\sstar m}$ and suppose that $n = |A|$.  Choose an injective function $\eta : A \rightarrow {}^m n$ such that, for all $a, a' \in A$ and for all $i < m$, $\eta(a)(i) < \eta(a')(i)$ if and only if $a <_i a'$.  For all $i < m$, $j < n$, and $a \in A$, notice that
 \[
  \Mon \models \varphi(\ob_{\eta(a)}, \oc_{i,j}) \mathiff \eta(a)(i) < j.
 \]
 Therefore, for all $i < m$, for all $<_i$-cuts $Y$ of $A$, there exists $\oc \in \Mon^{|\oz|}$ such that
 \[
  Y = \{ a \in A : \Mon \models \varphi(\ob_{\eta(a)}, \oc) \}.
 \]
 Consider the function $a \mapsto \ob_{\eta(a)}$ from $A$ to $\pi(\Mon)$.  By compactness, there exists a function $f : \Gamma \rightarrow \pi(\Mon)$ such that, for all $i < m$, for all $<_i$-cuts $Y$ of $\Gamma$, there exists $\oc \in \Mon^{|\oz|}$ such that
 \[
  Y = \{ a \in \Gamma : \Mon \models \varphi(f(a), \oc) \}.
 \]
 By Remark \ref{Rem_GenIndisc}, we can assume that $f : \Gamma \rightarrow \pi(\Mon)$ is a generalized indiscernible.  Therefore, for each $k < \omega$ and each quantifier-free $L_m$-type $p(x_0, \dots, x_{k-1})$, we have an associated $L$-type $p^*(\oy_0, \dots, \oy_{k-1})$ (over the same parameters as $\pi$ and $\varphi$) extending $\pi(\oy_0) \cup \dots \cup \pi(\oy_{k-1})$ such that, for all $\oa \in \Gamma^k$, if $\oa \models p$, then $f(\oa) \models p^*$.
 \if\papermode0
 { \color{violet}
 
  \textit{More Detail:} For each $i < m$, let $\cC_i$ be the set of all $<_i$-cuts of $\Gamma$.  Consider the type
  \begin{align*}
   & \Sigma( (\oy_a)_{a \in \Gamma}, (\oz_Y)_{Y \in \cC_i, i < m} ) = \\ & \ \bigcup \{ \pi(\oy_a) : a \in \Gamma \} \cup \\ & \ \{ \varphi(\oy_a, \oz_Y)^{\condiff a \in Y} : a \in \Gamma, Y \in \cC_i, i < m \}
  \end{align*}
  Fix a finite $\Sigma_0 \subseteq \Sigma$.  Let $A \subseteq \Gamma$ be finite such that only variables of the form $\oy_a$ for $a \in A$ are mentioned in $\Sigma_0$.  Then, we get that $\Sigma_0$ is consistent.  By compactness, $\Sigma$ is consistent, say witnessed by $(\od'_a)_{a \in \Gamma}$.  Let $f' : \Gamma \rightarrow \pi(\Mon)$ be given by $f'(a) = \od'_a$.
  
  Next, let $C$ be a small subset of $\Mon$ such that $\varphi$ and $\pi$ are over $C$.  Consider the type $\Sigma'$ which is $\Sigma$ together with
  \[
   \{ \psi(\oy_{\oa}) \leftrightarrow \psi(\oy_{\ob}) : k < \omega, \oa, \ob \in A^k, \qftp_{L_0}(\oa) = \qftp_{L_0}(\ob), \psi \in L(C) \}.
  \]
  For any finite $\Sigma'_0 \subseteq \Sigma'$, let $A \subseteq \Gamma$ be finite such that only variables of the form $\oy_a$ for $a \in A$ are mentioned in $\Sigma'_0$.  Moreover, let $\Psi(\oy_0, \dots, \oy_{k-1})$ be a finite set of $L(C)$-formulas such that only formulas of the form $\psi(\oy_{\oa}) \leftrightarrow \psi(\oy_{\ob})$ for $\psi \in \Psi$ and $\oa, \ob \in A^k$ are mentioned in $\Sigma'_0$.  Let $c : \Gamma^k \rightarrow {}^{\Psi_0} 2$ be given by
  \[
   c(\oa)(\psi) = 1 \mathiff \Mon \models \psi(f'(\oa)).
  \]
  Since $\LO^{\sstar m}$ has the Ramsey Property, there exists an embedding $g : A \rightarrow \Mon$ such that $c$ is monochromatic on $g(A)$.  Therefore, $g(A)$ realizes $\Sigma'_0$.  By compactness, $\Sigma'$ is consistent, say witnessed by $(\od_a)_{a \in \Gamma}$.  Let $f : \Gamma \rightarrow \pi(\Mon)$ be given by $f(a) = \od_a$.  This works.
 }
 \fi
 
 Since $T$ has NIP, $\varphi(\oy, \oz)$ has VC-dimension $< k$ for some $k < \omega$.  In other words,
 \[
  \Mon \models \neg \exists \oy_0 \dots \exists \oy_{k-1} \bigwedge_{s \in {}^k 2} \left( \exists \oz \bigwedge_{\ell < k} \varphi(\oy_\ell, \oz)^{s(\ell)} \right).
 \]
 
 For each $t \in {}^m 2$, define the quantifier-free $2$-$L_m$-type $p_t(x_0, x_1)$ as follows: $p_t \vdash x_0 \neq x_1$ and, for all $i < m$, $p_t \vdash x_0 <_i x_1$ if and only if $t(i) = 1$.  We can extend this to a quantifier-free $k$-$L_m$-type $q_t(x_0, \dots, x_{k-1})$ as follows: for all $\ell \neq \ell'$, $q_t \vdash x_\ell \neq x_{\ell'}$ and, for all $i < m$, for all $\ell < k - 1$, $q_t \vdash x_\ell <_i x_{\ell+1}$ if and only if $t(i) = 1$.
 
 Now fix $t, t' \in {}^m 2$ distinct.  We may assume, by perhaps swapping $t$ and $t'$, that there exists $i_0 < m$ such that $t(i_0) = 1$ and $t'(i_0) = 0$.  Since $\varphi$ has VC-dimension $< k$, there exists $s \in {}^k 2$ such that
 \[
   q_t^*(\oy_0, \dots, \oy_{k-1}) \vdash \neg \exists \oz \bigwedge_{\ell < k} \varphi(\oy_\ell, \oz)^{s(\ell)}.
 \]
 We define $q_r$ and $\sigma_r$ recursively as follows: Let $q_0 = q_t$ and $\sigma_0$ the identity permutation on $k$.  Fix $r \ge 0$ and assume that we have constructed $q_r$ and $\sigma_r$ such that
 \begin{equation}\label{Eq_qrandsigmar}
  q_r(x_0, \dots, x_{k-1}) \vdash x_{\sigma_r(0)} <_{i_0} x_{\sigma_r(1)} <_{i_0} \dots <_{i_0} x_{\sigma_r(k-1)}.
 \end{equation}  
 Then, choose $\ell_r < k-1$ minimal such that $s(\sigma_r(\ell_r)) = 0$ and $s(\sigma_r(\ell_r + 1)) = 1$.  Note that, if no such $\ell_r$ exists, then
 \[
  q_r^*(\oy_0, \dots, \oy_{k-1}) \vdash \exists \oz \bigwedge_{\ell < k} \varphi(\oy_{\sigma_r(\ell)}, \oz)^{s(\sigma_r(\ell))},
 \]
 since it is a $<_{i_0}$-cut.  In particular, $\ell_0$ exists.
 
 Let $\sigma_{r+1} = \sigma_r \circ (\ell_r \ \ell_r + 1)$ and let $q_{r+1}$ be $q_r$ except, for each $i < m$, we replace
 \[
  (x_{\ell_r} <_i x_{\ell_r+1})^{t(i)} \text{ with } (x_{\ell_r} <_i x_{\ell_r+1})^{t'(i)}.
 \]
 In particular, we maintain that $q_{r+1}$ and $\sigma_{r+1}$ satisfy \eqref{Eq_qrandsigmar}.  Terminate the construction when we first have
 \[
  q_{r+1}^*(\oy_0, \dots, \oy_{k-1}) \vdash \exists \oz \bigwedge_{\ell < k} \varphi(\oy_\ell, \oz)^{s(\ell)}
 \]

 Choose $\oa \in \Gamma^k$ such that $\oa \models q_r$.  Let
 \begin{align*}
  \psi_{t,t'}(\oy_0, \oy_1) := & \ \neg \exists \oz \bigl( \varphi(f(a_0), \oz)^{s(0)} \wedge \dots \wedge \varphi(f(a_{\ell_r-1}), \oz)^{s(\ell_r-1)} \wedge \\ & \varphi(\oy_0, \oz)^{s(\ell_r)} \wedge \varphi(\oy_1, \oz)^{s(\ell_r+1)} \wedge \\ & \varphi(f(a_{\ell_r+2}), \oz)^{s(\ell_r+2)} \wedge \dots \wedge \varphi(f(a_{k-1}), \oz)^{s(k-1)} \bigr).
 \end{align*}
 Consider the set
 \[
  \Gamma' = \{ a \in \Gamma : (a_0, \dots, a_{\ell_r-1}, a, a_{\ell_r+2}, \dots, a_{k-1}) \models q_r |_{x_0, \dots, x_{\ell_r}, x_{\ell_r+2}, \dots, x_{k-1}} \}.
 \]
 Since $\LO^{\sstar m}$ is \selfsim, $\Gamma' \cong \Gamma$.  Notice that, for all $a, b \in \Gamma'$,
 \begin{itemize}
  \item If $(a,b) \models p_t$, then $(a_0, \dots, a_{\ell_r-1}, a, b, a_{\ell_r+2}, \dots, a_{k-1}) \models q_r$, hence $\Mon \models \psi_{t,t'}(f(a), f(b))$.
  \item If $(a,b) \models p_{t'}$, then $(a_0, \dots, a_{\ell_r-1}, a, b, a_{\ell_r+2}, \dots, a_{k-1}) \models q_{r+1}$, hence $\Mon \models \neg \psi_{t,t'}(f(a), f(b))$.
 \end{itemize}
 Replace $\Gamma$ with $\Gamma'$ and repeat this process for all distinct $t, t' \in {}^m 2$.
 
 For $t \in {}^m 2$, let
 \[
  \psi_t(\oy_0, \oy_1) := \bigwedge_{t' \in {}^m 2, t' \neq t} \psi_{t,t'}(\oy_0, \oy_1).
 \]
 Finally, for $i < m$, let
 \[
  \psi_i(\oy_0, \oy_1) := \bigvee_{t \in {}^m 2, t(i) = 1} \psi_t(\oy_0, \oy_1).
 \]
 Then, it is clear that, for all $a, b \in \Gamma$ and $i < m$,
 \[
  a <_i b \text{ if and only if } \Mon \models \psi_i(f(a), f(b)).
 \]
 By Lemma \ref{Lem_FiniteConfig}, there exists an $\LO^{\sstar m}$-configuration into $\pi$.  Thus, $\Rk_{\LO}(\pi) \ge m$.
\end{proof}

Note that this proof uses generalized indiscernibles; this is the only such use in this paper.  In future work, we would like to remove the need for indiscernibility so that arguments such as these can be generalized to \fraisse classes without the Ramsey Property.

\begin{expl}[NIP]\label{Expl_LONIP}
 Suppose $T$ has NIP.  Then $\LO$-rank is precisely op-dimension.  In particular, $\LO$-rank is additive (see Theorem 2.2 of \cite{gh2}).
 
 If $T$ is distal, then op-dimension coincides with dp-rank (see Remark 3.2 of \cite{gh2}).  Therefore, for distal $T$, $\LO$-rank is dp-rank.
\end{expl}

In the next example, we show that $\LO$-rank can jump from $0$ to $\infty$ in a theory with the independence property.

\begin{expl}
 Let $L$ consist of infinitely many binary relation symbols $R_i$ for $i < \omega$ and let $T$ be the model companion of the theory of $L$-hypergraphs.  Then,
 \[
  \Rk_\LO(1) = 0 \text{ and } \Rk_\LO(2) = \infty.
 \]
\end{expl}
 
\begin{proof}
 Let $\Mon$ be a monster model for $T$ and let $\Gamma$ be the \fraisse limit of $\LO$.  Towards a contradiction, suppose there exists an $\LO$-configuration into $\Mon$ over some small $C \subseteq \Mon$.  Since $\LO$ is \selfsim, there exists a function $f : \Gamma \rightarrow \Mon$ satisfying the conclusion of Proposition \ref{Prop_WeakGoodEmbedding}.  Thus, since $L$ is a binary language, for all $a, b \in \Gamma$, the type $\tp_L(f(a), f(b) / C)$ is determined by the type $\tp_L(f(a), f(b))$.  On the other hand, since $T$ is symmetric,
 \[
  \tp_L(f(a), f(b)) = \tp_L(f(b), f(a)).
 \]
 Therefore, $\Gamma \models a < b$ if and only if $\Gamma \models b < a$, a contradiction.

 On the other hand, fix $m < \omega$.  For each $i < m$, let
 \[
  I(<_i)(y_{0,0}, y_{0,1}, y_{1,0}, y_{1,1}) = R_i(y_{0,0}, y_{1,1}).
 \]
 For each $A \in \LO^{\sstar m}$, there exists a function $f_A : A \rightarrow \Mon^2$ such that, for all $a, b \in A$ and $i < m$,
 \[
  A \models a <_i b \mathiff \Mon \models I(<_i)(f_A(a), f_A(b)).
 \]
 \if\papermode0
 { \color{violet}
 
  \emph{More Detail:} To see this, first endow $A \times 2$ with a $L$-structure by setting, for all $a, b \in A$ and $i < m$,
   \[
    R_i((a, 0), (b, 1)) \text{ if } a <_i b.
   \]
   Then, there exists an embedding $g : A \times 2 \rightarrow \Mon$; let $f(a) = (g(a,0), g(a,1))$.  Then, it is clear that, for all $a, b \in A$ and $i < m$,
  \[
   A \models a <_i b \mathiff \Mon \models \varphi_i(f(a), f(b)).
  \]
 }
 \fi
 Hence, $(I, f_A)_{A \in \LO^{\sstar m}}$ is an $\LO^{\sstar m}$-configuration into $\Mon^2$.
\end{proof}

%%%%%%%%%%%%%%%%%%%%%%%%%%%%%%%%%%%%%%%%%%
%% Subsection -- Equivalence Class Rank %%
%%%%%%%%%%%%%%%%%%%%%%%%%%%%%%%%%%%%%%%%%%

\subsection{Equivalence Class Rank}\label{Subsect_EQR}

For this subsection, we consider the \algtriv \fraisse class $\EE$.  For any $m \ge 1$, let $L_m$ be the language of $\EE^{\sstar m}$, which consists of $m$ binary relation symbols, $E_i$ for $i < m$.

It turns out that $\EE$-rank is bounded above by dp-rank.

\begin{prop}\label{Prop_EEUpperBound}
 For any partial type $\pi$ in any theory $T$,
 \[
  \Rk_\EE(\pi) \le \dpRk(\pi).
 \]
\end{prop}

\begin{proof}
 This follows similarly to the proof of Proposition \ref{Prop_opDimUpperBound}.  Suppose $\Rk_\EE(\pi) \ge m$.  Let $\Gamma$ be the \fraisse limit of $\EE^{\sstar m}$ and let $\Mon$ be a monster model for $T$.  Let $(I, f_A)_{A \in \EE^{\sstar m}}$ be an $\EE^{\sstar m}$-configuration into $\Mon$.  So, for all $i < m$, for all $A \in \EE^{\sstar m}$, for all $a, b \in A$,
 \[
  A \models E_i(a,b) \mathiff \Mon \models I(E_i)(f_A(a), f_A(b)).
 \]
 Fix $n < \omega$.  Create $A \in \EE^{\sstar m}$ with universe ${}^m n$ by setting, for all $i < m$ and $g, h \in A$,
 \[
  A \models E_i(g, h) \mathiff g(i) = h(i).
 \]
 For $g \in A$, let $\oc_g = f_A(g)$.  Thus, for all $i < m$, for all $g, h \in A$,
 \[
  \Mon \models I(E_i)(\oc_g, \oc_h) \mathiff g(i) = h(i).
 \]
 For each $j < n$, overloading notation, let $j$ denote the function from $m$ to $n$ that is constantly $j$.  Then, for each $g \in {}^m n$, we have that
 \[
  \pi(\oy) \cup \{ I(E_i)(\oy, \oc_j)^{\condiff g(i) = j} : i < m, j < n \}
 \]
 is consistent (realized by $\oc_g$).  This is an ICT-pattern of depth $m$ and length $n$ in $\pi$.  Since $n$ was arbitrary, by compactness, $\pi$ has dp-rank $\ge m$.
\end{proof}

Moreover, $\EE$-rank is bounded below by the dimension of the target type (assuming the target theory has infinite models).

\begin{prop}\label{Prop_EELowerBound}
 For all theories $T$ with infinite models,
 \[
  \Rk_{\EE}(m) \ge m.
 \]
\end{prop}

\begin{proof}
 Since $T$ has infinite models, there exists an injective function $g : \omega \rightarrow \Mon$.  Fix a tuple of variables $\oy$ and let $m = |\oy|$.  Fix $A \in \EE^{\sstar m}$ and choose $n < \omega$ such that $A$ embeds into $n^{m+1}$ viewed as an element of $\EE^{\sstar m}$ as in Lemma \ref{Lem_EEmuniverse}.
 \if\papermode0
 { \color{violet}
 
  \emph{More Detail:} To see this, for each $i < m$, enumerate the $E_i$-classes of $A$ (say $a$ belongs to the $f_i(a)$th $E_i$-class) and, for each $(E_0 \cap \dots \cap E_{m-1})$-class, enumerate the elements of that class (say $a$ is the $f_m(a)$th element of its $(E_0 \cap \dots \cap E_{m-1})$-class).  Let $n$ be the maximum of $f_i(a)+1$ over all $i \le m$ and $a \in A$.  Then, the map $f : A \rightarrow n^{m+1}$ given by $f(a) = (f_0(a), \dots, f_m(a))$ is such an embedding. 
  
 }
 \fi
 Thus, we may assume $A$ is this $L_m$-structure on $n^{m+1}$.  Let $f_A : A \rightarrow \Mon^m$ be given by $f_A(\oa) = (g(a_0), \dots, g(a_m))$ for each $\oa \in n^{m+1}$.
  
 For each $i < m$, let
 \[
  I(E_i)(y_{0,0}, \dots, y_{0,m}, y_{1,0}, \dots, y_{1,m}) = \left[ y_{0,i} = y_{1,i} \right].
 \]
 It is easy to check that $(I, f_A)_{A \in \EE^{\sstar m}}$ is an $\EE^{\sstar m}$-configuration into $\oy = \oy$.  Thus, $\Rk_{\EE}(m) \ge m$.
\end{proof}

We say that a theory $T$ is \emph{dp-minimal} if $\dpRk(y = y) = 1$ for some (any) single variable $y$ in $T$. Combining the previous two results, we conclude that $\EE$-rank is precisely equal to the dimension of the target type in dp-minimal theories.

\begin{cor}\label{Cor_EEDpRank}
 Let $T$ be a dp-minimal theory with infinite models.  Then, $\Rk_\EE(m) = m$.
\end{cor}

\begin{proof}
 Let $\oy$ be an $m$-tuple of variables from $T$.  By Proposition \ref{Prop_EELowerBound},
 \[
  \Rk_{\EE}(\oy = \oy) \ge | \oy |.
 \]
 Since dp-rank is subadditive \cite{kou}, $\dpRk(\oy = \oy) \le |\oy|$.  So, by Proposition \ref{Prop_EEUpperBound},
 \[
  \Rk_{\EE}(\oy = \oy) \le \dpRk(\oy = \oy) \le |\oy|.
 \]
 Thus, $\Rk_\EE(\oy = \oy) = \dpRk(\oy = \oy) = |\oy|$.
\end{proof}

\begin{ques}
 If $T$ is dp-minimal and $\pi$ is a partial type in $T$, then does $\Rk_{\EE}(\pi) = \dpRk(\pi)$?  More generally, under what conditions does $\Rk_{\EE}(\pi) = \dpRk(\pi)$?
\end{ques}

Although this question is still open, we have examples where $\EE$-rank and dp-rank differ, even in an NIP theory.

\begin{expl}\label{Expl_EENotDp}
 Fix $k \ge 2$ and let $T$ be the theory of the \fraisse limit of $\LO^{\sstar k}$.  We claim that, in the theory $T$,
 \[
  \left\lfloor \frac{k}{2} \right\rfloor \le \Rk_\EE(1) < k.
 \]
 (On the other hand, $\dpRk(y = y) = k$, so these ranks disagree.)
\end{expl}

\begin{proof}
 Let $m = \lfloor k/2 \rfloor$ and fix $A \in \EE^{\sstar m}$.  As in the proof of Proposition \ref{Prop_EELowerBound}, there exists $n < \omega$ such that $A$ embeds into $X = n^{m+1}$ with $L_m$-structure given in Lemma \ref{Lem_EEmuniverse}.  For each $i < m$, we define two linear orders $<_{2i}$ and $<_{2i+1}$ on $X$ as follows: for all $\oa, \ob \in X$, let
 \begin{itemize}
  \item $\oa <_{2i} \ob$ if $a_i < b_i$ or $a_i = b_i$ and $a_{j_0} < b_{j_0}$ where $j_0 = \min \{ j < n : a_j \neq b_j \}$, and
  \item $\oa <_{2i+1} \ob$ if $a_i > b_i$ or $a_i = b_i$ and $a_{j_0} < b_{j_0}$ where $j_0 = \min \{ j < n : a_j \neq b_j \}$.
 \end{itemize}
 It is clear from definition that, for all $i < m$, for all $\oa, \ob \in X$,
 \[
  E_i(\oa, \ob) \mathiff (\oa <_{2i} \ob \leftrightarrow \oa <_{2i+1} \ob).
 \]
 If $k = 2m + 1$, then define $<_{2m}$ arbitrarily.  Since $(X, <_i)_{i < k}$ is an element of $\LO^{\sstar k}$, we get an embedding of $X$ into $\Mon$, the monster model of $T$.  Composing this function with the one sending $A$ to $X$, we get a function $f_A : A \rightarrow \Mon$ such that, for all $a, b \in A$ and $i < m$,
 \[
  A \models E_i(a, b) \mathiff (f_A(a) <_{2i} f_A(b) \leftrightarrow f_A(a) <_{2i+1} f_A(b)).
 \]
 For each $i < m$, let
 \[
  I(E_i)(y_0, y_1) := \left[ y_0 <_{2i} y_1 \leftrightarrow y_0 <_{2i+1} y_1 \right]
 \]
 Then, $(I, f_A)_{A \in \EE^{\sstar m}}$ is an $\EE^{\sstar m}$-configuration into $\Mon$.  Thus, $\Rk_{\EE}(1) \ge m$.
 
 On the other hand, suppose that $\Gamma$ is the \fraisse limit of $\EE^{\sstar k}$, $\Mon$ is a monster model of $T$, and there exists an $\EE^{\sstar k}$-configuration over a finite $C$.  Since $\EE^{\sstar k}$ is age indivisible, there exists a function $f : \Gamma \rightarrow \Mon$ satisfying the conclusion of Proposition \ref{Prop_WeakGoodEmbedding}.  Fix $a \in \Gamma$ and, for each $s \in {}^k 2$, choose $b_s \in \Gamma$ such that, for all $i < k$, $E_i(a, b_s)$ if and only if $s(i) = 1$.  Consider the $2$-types in $T$:
 \[
  p_{s,0} = \tp_L(f(a), f(b_s)) \text{ and } p_{s,1} = \tp_L(f(b_s), f(a)).
 \]
 Since $L$ has only binary relations and each $f(a)$ and $f(b_s)$ have the same $L$-type over $C$, these types determine the $L$-types over $C$.  For any $s$ not the identically $1$ function, $f(a) \neq f(b_s)$.  Otherwise, suppose that $f(a) = f(b_s)$ where $s \in {}^k 2$ with $s(i) = 0$.  Then, we get that
 \[
  \tp_L(f(a), f(b_1) / C) = \tp_L(f(b_s), f(b_1) / C)
 \]
 (where $1$ is the identically $1$ function).  Since $\Gamma \models E_i(a, b_1)$, this implies that $\Gamma \models E_i(b_s, b_1)$, hence $\Gamma \models E_i(a, b_s)$, which is a contradiction.  Similarly, for $s$ and $t$ not identically $1$, if $(s,i) \neq (t,j)$, then $p_{s,i} \neq p_{t,j}$.  Therefore, we have at least $2 \cdot (2^k - 1)$ many non-equality $2$-types in $T$.  On the other hand, there are $2^k$ many non-equality $2$-types in $T$ (one type for each possible assignment of $x <_i y$ or $x >_i y$ for all $i < k$).  Thus, $2^{k+1} - 2 \le 2^k$, a contradiction.  Therefore, there is no $\EE^{\sstar k}$-configuration into $y = y$.  Thus, $\Rk_{\EE}(y = y) < k$.
\end{proof}

\if\papermode0
{ \color{violet}

\textit{Bonus Example: $\EE$-rank is discontinuous.}

%% Discontinuity of \EE-rank %%
\begin{expl}
 Let $T$ be the theory of infinitely many independent equivalence relations in the language $L$ consisting of binary relation symbols $E_i$ for $i < \omega$.  Fix $c \in \Mon$ and let $p(y)$ be the type generated by $\{ E_i(y, c) : i < \omega \}$.  Then, $\Rk_\EE(p) = 1$ but $\Rk_\EE(\psi) = \infty$ for all $\psi \in p$.  This shows that, generally, we cannot have
 \[
  \Rk_\KK(p) = \min \{ \Rk_\KK(\psi) : \psi \in p \}
 \]
 (even when $p$ is over a small subset of $\Mon$).
\end{expl}

\begin{proof}
 Clearly $\Rk_\EE(p) \ge 1$.  Towards a contradiction, suppose $\Rk_\EE(p) \ge 2$.  Let $\Gamma$ be the \fraisse limit of $\EE^{\sstar 2}$, let $f : \Gamma \rightarrow p(\Mon)$, and let $\varphi_t(y_0, y_1)$ for $t < 2$ be $L(\Mon)$-formulas such that, for all $a, b \in \Gamma$ and $t < 2$,
 \[
  \Gamma \models E_t(a, b) \mathiff \Mon \models \varphi_t(f(a), f(b)).
 \]
 (i.e., $f$ is an $\EE^{\sstar 2}$-configuration into $p$.)  Let $C \subseteq \Mon$ be the finite set of parameters used in $\varphi_0$ and $\varphi_1$.  For all $a, b \in \Gamma$ and $t < 2$, if $f(a) = f(b)$, then, since $\Gamma \models E_t(a, a)$, hence $\Mon \models \varphi_t(f(a), f(a))$, we have that $\Mon \models \varphi_t(f(a), f(b))$, hence $\Gamma \models E_t(a, b)$.  Fix distinct $a_0, a_1, a_2 \in \Gamma$ such that
 \[
  \Gamma \models E_0(a_0, a_1) \wedge \neg E_1(a_0, a_1) \wedge \neg E_0(a_0, a_2).
 \]
 Thus, $f(a_0) \neq f(a_1)$ and $f(a_0) \neq f(a_2)$.  Then, since $p(\Mon)$ has no structure
 \[
  \tp_L(f(a_0), f(a_1) / C) = \tp_L(f(a_0), f(a_2) / C).
 \]
 Thus, since $\Gamma \models E_0(a_0, a_1)$, we have that $\Mon \models \varphi_0(f(a_0), f(a_1))$, hence $\Mon \models \varphi_0(f(a_0), f(a_2))$, hence $\Gamma \models E_0(a_0, a_2)$.  This is a contradiction.  Therefore, $\Rk_\EE(p) = 1$.
 
 Fix $m < \omega$ and fix $\psi(y) \in p(y)$.  Suppose that $I \subseteq \omega$ is infinite such that $\psi$ has no mention of $E_i$ for all $i \in I$.  Therefore, in the language with signature $\{ E_i : i \in I \}$, $\psi(\Mon)$ has infinitely many independent equivalence relations.  Choose some $\sigma : m \rightarrow I$ injective.  Define the map $I$ by setting, for all $i < m$, $I(E_i) = E_{\sigma(i)}$.  It is clear that $I$ satisfies the conditions of Lemma \ref{Lem_FiniteConfig}.  Therefore, there exists an $\EE^{\sstar m}$-configuration into $\psi$.  Since $m$ was arbitrary, $\Rk_\EE(\psi) = \infty$.
\end{proof}

}
\fi

%%%%%%%%%%%%%%%%%%%%%%%%%%%%%%
%% Subsection -- Graph Rank %%
%%%%%%%%%%%%%%%%%%%%%%%%%%%%%%

\subsection{Graph Rank}\label{Subsect_GR}

For this subsection, we consider the \algtriv \fraisse class $\GG$.

\begin{expl}[NIP]\label{Expl_GGNIP}
 If $T$ is a theory with NIP, then, for all types $\pi$ in $T$, $\Rk_\GG(\pi) = 0$.  Thus, $\GG$-rank is trivially additive.
\end{expl}

\begin{proof}
 This follows from Theorem \ref{Thm_CCExamples} (2) and Corollary \ref{Cor_StarSim}.
\end{proof}

Similar to $\LO$-rank, the $\GG$-rank can jump from $0$ to $\infty$ in a theory with the independence property.

 %% Vince, 11/11/2022: Modified to the reviewer's specification.
\begin{expl}
 Let $L$ be the language consisting of binary relation symbols $R_i$ for $i < \omega$ and let $T$ be the model companion of the theory of $L$-structures for which $R_i$ is a triangle-free graph for all $i < \omega$.  Then,
 \[
  \Rk_\GG(1) = 0 \text{ and } \Rk_\GG(2) = \infty.
 \]
\end{expl}

\begin{proof}
 Let $\Mon$ be a monster model for $T$.  Towards a contradiction, suppose there exists $(I, f_A)_{A \in \GG}$ a $\GG$-configuration into $\Mon$.  Fix $n < \omega$ such that $I(E)$ mentions only $R_i$ for $i < n$.  By quanitifier elimination, there exists $S \subseteq {}^n 2$ such that
 \[
  I(E)(y_0, y_1) = \bigvee_{s \in S} \bigwedge_{i < n} R_i(y_0, y_1)^{s(i)}.
 \]
 By swapping $E$ with $\neg E$ if necessary, we may assume the constant zero function is not in $S$.  If we consider a finite complete graph $A$, then $f_A(A)$ can be viewed as a complete graph with edge colors in $S$.  By Ramsey's Theorem, for sufficiently large $A$, there exists a triangle of a fixed color $s_0 \in S$.  By assumption, there exists $i_0 < n$ such that $s_0(i_0) = 1$, so this is an $R_{i_0}$-triangle.  This is a contradiction.

 Fix an arbitrary $m < \omega$ and define, for each $i < m$, $L$-formulas as follows:
 \[
  I(E_i)(y_{0,0}, y_{0,1}, y_{1,0}, y_{1,1}) = R_i(y_{0,0}, y_{1,1}) \wedge R_i(y_{0,1}, y_{1,0}).
 \]
 For any $A \in \GG^{\sstar m}$, there exists a function $f_A : A \rightarrow \Mon^2$ such that, for all $a, b \in A$ and $i < m$,
 \[
  A \models E_i(a,b) \mathiff \Mon \models I(E_i)(f_A(a), f_A(b)).
 \]
 That is, $(I, f_A)_{A \in \GG^{\sstar m}}$ is a $\GG^{\sstar m}$-configuration into $\Mon^2$.
\end{proof}

%%%%%%%%%%%%%%%%%%%%%%%%%%%%%%%%%%%%%%%%%
%% Subsection -- Into the Random Graph %%
%%%%%%%%%%%%%%%%%%%%%%%%%%%%%%%%%%%%%%%%%

\subsection{Into the Random Graph}\label{Subsect_RandGraph}

In this subsection, we study the specific case where the target theory is the theory of the random graph.  It turns out that $\KK$-rank, for various examples of $\KK$, acts in an interesting manner in this theory.

For this subsection, let $T$ be the theory of the random graph in the language $L$ with a single binary relation $R$ and let $\Mon$ be a monster model for $T$.  Let $\KK$ be \analgtriv \fraisse class in a language $L_0$ with a single binary relation symbol, $E$.  For any $m \ge 1$, let $L_m$ be the language of $\KK^{\sstar m}$, which consists of $m$ binary relation symbols; call them $E_i$ for $i < m$.  Let $\Gamma$ be the \fraisse limit of $\KK^{\sstar m}$.

Fix $n \ge 1$ and, for each $t < 2$, consider the set $X_t = n \times \{ t \}$.  Let $\cG^n$ be the set of bipartite graphs with parts $X_0$ and $X_1$.  Then, $| \cG^n | = 2^{n^2}$.  Say that $G = (X_0 \cup X_1, F) \in \cG^n$ is \emph{symmetric} if, whenever $\{ (i,0),(j,1) \} \in F$, $\{ (j,0), (i,1) \} \in F$.  Let $\cGs^n$ be the set of symmetric bipartite graphs with parts $X_0$ and $X_1$ and let $\cGns^n$ be the set of non-symmetric bipartite graphs with parts $X_0$ and $X_1$.  Notice that $|\cGs^n| = 2^{\binom{n+1}{2}}$.  To see this, observe that, for each $i \le j < n$, we can choose whether or not to put $\{ (i,0),(j,1) \}, \{ (j,0), (i,1) \} \in F$.  This gives us $\binom{n}{2} + n = \binom{n+1}{2}$ choices.  Thus,
\[
 | \cGns^n | = 2^{n^2} - 2^{\binom{n+1}{2}} = 2^{\binom{n+1}{2}} \left( 2^{\binom{n}{2}} - 1 \right).
\]
For each $G \in \cG^n$, let $G^*$ be the graph where we ``swap parts'' (i.e., $\{ (i,0), (j,1) \}$ is an edge of $G$ if and only if $\{ (j,0), (i,1) \}$ is an edge of $G^*$).  Clearly $(G^*)^* = G$ and, for all $G \in \cG^n$, $G \in \cGs^n$ if and only if $G^* = G$.

Let $\cS_2$ denote the set of all quantifier-free $2$-$L_m$-types over $\emptyset$, $p(x_0, x_1)$, such that $(a_0, a_1) \models p$ for some distinct $a_0, a_1 \in \Gamma$.  For $p \in \cS_2$, let $p^*$ be the type in $\cS_2$ such that, for all $i < m$,
\[
 p^*(x_0, x_1) \vdash E_i(x_0, x_1) \mathiff p(x_0, x_1) \vdash E_i(x_1, x_0).
\]

The next three propositions give conditions on $\KK$ that guarantee that $\Rk_\KK(n) = n^2 - 1$ for $n \ge 2$.  These conditions are met by $\LO$, $\GG$, and $\TT$.

\begin{prop}\label{Prop_KKquad1}
 Fix $n \ge 1$.  Assume that $\KK$ is either reflexive or irreflexive.  Assume also that $\KK$ is either symmetric or trichotomous.  Then,
 \[
  \Rk_\KK(n) \ge n^2 - 1.
 \]
 Moreover, this is witnessed by a parameter-free configuration.
\end{prop}

\begin{proof}
 We may assume $n \ge 2$, since the statement is trivial when $n = 1$.
 
 Let $m = n^2 - 1$ and let $\oy$ be an $n$-tuple of variables from $T$.  Consider the function $g : \cS_2 \rightarrow {}^m 2$ where, for all $p \in \cS_2$ and $i < m$,
 \[
  g(p)(i) = 1 \mathiff p(x_0, x_1) \vdash E_i(x_0, x_1).
 \]
 Since $\KK$ is symmetric or trichotomous and $\KK$ is reflexive or irreflexive, $g$ is injective.  Thus, $|\cS_2| \le 2^m$.
 \if\papermode0
 { \color{violet}
 
  \emph{More Detail:} Fix $p \in \cS_2$ and $i < m$.  If $\KK$ is symmetric, then $g(p)(i) = 1$ if and only if $p \vdash E_i(x_0, x_1)$ if and only if $p \vdash E_i(x_1, x_0)$.  If $\KK$ is trichotomous, then $g(p)(i) = 1$ if and only if $p \vdash E_i(x_0, x_1)$ if and only if $p \vdash \neg E_i(x_1, x_0)$.  If $\KK$ is reflexive, then $p \vdash E_i(x_0, x_0) \wedge E_i(x_1, x_1)$.  If $\KK$ is irreflexive, then $p \vdash \neg E_i(x_0, x_0) \wedge \neg E_i(x_1, x_1)$.  In any case, we see that $g$ is injective.  Thus,
  \[
   \left| \cS_2 \right| \le 2^m.
  \]
 }
 \fi
 
 Choose any injective function $h : \cS_2 \rightarrow \Pow(\cG^n)$ with the following conditions:
 \begin{enumerate}
  \item If $\KK$ is symmetric, then, for all $p \in \cS_2$, $h(p) = \{ G, G^* \}$ for some $G \in \cG^n$.
  \item If $\KK$ is trichotomous, then, for all $p \in \cS_2$, $h(p) = \{ G \}$ for some $G \in \cGns^n$ and $h(p^*) = (h(p))^*$.
 \end{enumerate}
 To see that such an $h$ exists, we have to consider two cases:
 
 \begin{caseone}
  $\KK$ is symmetric.
 \end{caseone}
 
 In this case, the number of allowed outputs for $h$ is
 \begin{align*}
  | \cGs^n | + \frac{1}{2} | \cGns^n | = & \ 2^{\binom{n+1}{2}} + \frac{1}{2} \left( 2^{n^2} - 2^{\binom{n+1}{2}} \right) = \\ & \ 2^{n^2-1} + 2^{\binom{n+1}{2} - 1} \ge 2^{n^2 - 1}.
 \end{align*}
 
 \begin{casetwo}
  $\KK$ is trichotomous.
 \end{casetwo}
 
 In this case, the number of allowed outputs for $h$ is
 \[
  | \cGns^n | = 2^{\binom{n+1}{2}} \left( 2^{\binom{n}{2}} - 1 \right) \ge 2^{n^2 - 1}.
 \]
 
 In either case, the number of allowed outputs for $h$ is at least
 \[
  2^{n^2-1} = 2^m \ge | \cS_2 |.
 \]
 Therefore, such a function $h$ exists.
  
 For each $G = (X_0 \cup X_1, F) \in \cG^n$, let
 \[
  \varphi_G(y_{0,0}, \dots, y_{0,n-1}, y_{1,0}, \dots, y_{1,n-1}) = \bigwedge_{ i, j < n } R(y_{0,i}, y_{1,j})^{\condiff \{ (i,0), (j,1) \} \in F}.
 \] 
 Finally, for each $i < m$, let
 \[
  I(E_i)(\oy_0, \oy_1) = \bigvee_{p \in \cS_2, p(x_0, x_1) \vdash E_i(x_0, x_1)} \left( \bigvee_{G \in h(p)} \varphi_G(\oy_0, \oy_1) \right).
 \]
 For any $A \in \KK^{\sstar m}$, consider the set $n \times A$ and endow it with an $R$-graph structure as follows: For all distinct $a, b \in A$, choose some $G = (X_0 \cup X_1, F) \in h(\qftp_{L_0}(a,b))$.  For all $i, j < n$, define
 \[
  R((i, a), (j, b)) \mathiff \{ (i,0), (j,1) \} \in F
 \]
 and add no other $R$-edges among $n \times \{ a, b \}$.  Since $\Mon$ is a model of the random graph, this $R$-structure on $n \times A$ embeds into $\Mon$, say via $g : (n \times A) \rightarrow \Mon$.  Define $f_A : A \rightarrow \Mon^n$ by setting $f_A(a) = ( g(i,a) )_{i < n}$.  Then, it is clear that, for all $a, b \in A$ and $i < m$,
 \[
  A \models E_i(a,b) \mathiff \models I(E_i)(f_A(a), f_A(b)).
 \]
 Thus, $(I, f_A)_{A \in \KK^{\sstar m}}$ is a $\KK^{\sstar m}$-configuration into $\oy = \oy$.  Moreover, notice that this gives us a parameter-free configuration.
\end{proof}

\begin{prop}\label{Prop_KKquad2}
 Fix $n \ge 1$.  Assume that $\KK$ is \selfsim\ and \generic.  Then,
 \[
  \Rk_\KK(n) \le n^2.
 \]
\end{prop}

\begin{proof}
 Fix $m \ge 1$ and let $\oy$ be an $n$-tuple of variables from $T$.  Suppose that there exists a $\KK^{\sstar m}$-configuration into $\oy = \oy$.  By Proposition \ref{Prop_SelfSimStar}, since $\KK$ is \selfsim, $\KK^{\sstar m}$ is \selfsim.  Since $\KK^{\sstar m}$ is \selfsim, we may assume that there exist $I$, $f$, and $J$ satisfying the conclusion of Proposition \ref{Prop_GoodEmbedding}.  That is, $I : \sig(L_0) \rightarrow L(\Mon)$, $f : \Gamma \rightarrow \Mon^n$, and $J \subseteq |\oy|$ are such that
 \begin{enumerate}
  \item for all $i < m$, for all $a,b \in \Gamma$, $\Gamma \models E_i(a,b)$ if and only if $\Mon \models I(E_i)(f(a), f(b))$;
  \item for all $a, b \in \Gamma$, $\tp_L(f(a) / C) = \tp_L(f(b) / C)$;
  \item for all $j \in J$ and all $a, b \in \Gamma$, $f(a)_j = f(b)_j$; and
  \item for all $p \in \cS_2$, there exist $a_p, b_p \in \Gamma$ such that $(a_p, b_p) \models p$ and, for all $i, j \in |\oy| \setminus J$, $f(a_p)_i \neq f(b_p)_j$.
 \end{enumerate}
 Since $L$ is a binary language, condition (2) tells us that the type $\tp_L(f(a_p), f(b_p))$ determines the type $\tp_L(f(a_p), f(b_p) / C)$.
 
 We get a function $h : \cS_2 \rightarrow \cG^n$ as follows:  Fix $p \in \cS_2$.  For all $i, j < n$, we put $\{ (i,0), (j,1) \}$ in the edge set of $h(p)$ if and only if $R(f(a_p)_i, f(b_p)_j)$.  If we have $p, p' \in \cS_2$ distinct, then $\tp_L(f(a_p), f(b_p))$ and $\tp_L(f(a_{p'}), f(b_{p'}))$ disagree on some formula of the form $R(y_{0,i}, y_{1,j})$.  Therefore, $h(p) \neq h(p')$.  Hence, $h$ is injective.
 
 Since $\KK$ is \generic, by Proposition \ref{Prop_GenericStar}, $\KK^{\sstar m}$ is \generic.  Thus, $| \cS_2 | \ge 2^m$.  On the other hand, $| \cG^n | = 2^{n^2}$.  Therefore, $2^m \le 2^{n^2}$.  Thus, $m \le n^2$.
\end{proof}

\begin{prop}\label{Prop_KKquad3}
 Assume that $\KK$ is either reflexive or irreflexive.  Assume that $\KK$ is \selfsim\ and \generic.
 \begin{enumerate}
  \item If $\KK$ is trichotomous, then $\Rk_\KK(n) < n^2$ if $n \ge 1$.
  \item If $\KK$ is symmetric, then $\Rk_\KK(n) < n^2$ if $n \ge 2$.
 \end{enumerate}
\end{prop}

\begin{proof}
 Fix $m \ge 1$ and $n \ge 1$.  Let $h : \cS_2 \rightarrow \cG^n$ be the injective function from the proof of Proposition \ref{Prop_KKquad2} and consider the map $p \mapsto (a_p, b_p)$ from that proof.  We will show that $h$ is not surjective.
 
 \begin{caseone}
  $\KK$ is trichotomous.
 \end{caseone}
 
 If $p \in \cS_2$, then
 \[
  \tp_L(f(a_p), f(b_p)) \neq \tp_L(f(b_p), f(a_p)).
 \]
 Otherwise, $\tp_L(f(a_p), f(b_p) / C) = \tp_L(f(b_p), f(a_p) / C)$, hence $E_i(a_p,b_p)$ if and only if $E_i(b_p,a_p)$ for all $i < m$, which is a contradiction.  Thus, $h(p) \in \cGns^n \subsetneq \cG^n$.
 
 \begin{casetwo}
  $\KK$ is symmetric and $n \ge 2$.
 \end{casetwo}
 
 If $p, p' \in \cS_2$ are distinct, then
 \[
  \tp_L(f(a_p), f(b_p)) \neq \tp_L(f(b_{p'}), f(a_{p'})).
 \]
 Otherwise, $\tp_L(f(a_p), f(b_p) / C) = \tp_L(f(b_{p'}), f(a_{p'}) / C)$, hence $E_i(a_p,b_p)$ if and only if $E_i(b_{p'},a_{p'})$ if and only if $E_i(a_{p'},b_{p'})$ for all $i < m$, which is a contradiction.  Since $n \ge 2$, $\cGns^n \neq \emptyset$.  If $p, p' \in \cS_2$ are distinct and $h(p) \in \cGns^n$, then $h(p) \neq (h(p'))^*$.
 
 In either case, we see that $h$ is not surjective.  Therefore,
 \[
  2^m \le | \cS_2 | < | \cG^n | = 2^{n^2}.
 \]
 So $m < n^2$.
\end{proof}

We apply Propositions \ref{Prop_KKquad1}, \ref{Prop_KKquad2}, and \ref{Prop_KKquad3} to $\LO$, $\GG$, and $\TT$.

\begin{expl}[$\KK = \LO$]\label{Expl_LORandGr} 
 For all $n \ge 1$,
 \[
  \Rk_\LO(n) = n^2 - 1.
 \]
\end{expl}
 
\begin{proof}
 Since $\LO$ is irreflexive and trichotomous, Proposition \ref{Prop_KKquad1} gives us that $\Rk_\LO(n) \ge n^2 - 1$.  Moreover, $\LO$ is \selfsim\ and \generic, so Proposition \ref{Prop_KKquad3} gives us that $\Rk_\LO(n) < n^2$.
\end{proof}

\begin{expl}[$\KK = \TT$]\label{Expl_FFandGr}
 For all $n \ge 1$,
 \[
  \Rk_\TT(n) = n^2 - 1.
 \]
\end{expl}

\begin{proof}
 Similar to Example \ref{Expl_LORandGr}.
\end{proof}

In this paper, although the focus is not on $\TT$-rank, we do get this result ``for free.''  In future work, we may examine $\TT$-rank in other contexts.

\begin{expl}[$\KK = \GG$]\label{Expl_GRandGr}
 For all $n \ge 1$,
 \[
  \Rk_\GG(n) = \begin{cases} 1 & \text{ if } n = 1, \\ n^2 - 1 & \text{ if } n \ge 2 \end{cases}.
 \]
\end{expl}

\begin{proof}
 Since $\GG$ is irreflexive and symmetric, Proposition \ref{Prop_KKquad1} gives us that $\Rk_\GG(n) \ge n^2 - 1$.  Moreover, $\GG$ is \selfsim\ and \generic, so Proposition \ref{Prop_KKquad3} gives us that $\Rk_\GG(n) < n^2$ for $n \ge 2$.  To see that $\Rk_\GG(1) = 1$, use Proposition \ref{Prop_KKquad2} and Lemma \ref{Lem_IdentityConfig}.
\end{proof}

We now turn our attention to when $\KK = \EE$.  This will take more work because $\EE$ is not \selfsim.

For $t \in {}^m 2$ and $\oa, \ob \in \omega^m$, we say that $\oa \le_t \ob$ if, for all $i < m$,
\begin{itemize}
 \item $a_i < b_i$ if $t(i) = 1$ and
 \item $a_i = b_i$ if $t(i) = 0$.
\end{itemize}
Observe that, for any $\oa, \ob \in \omega^m$, there exists at most one $t \in {}^m 2$ such that $\oa \le_t \ob$.

For positive integers $m$ and $n$, consider the following property, which turns out to be implied by the existence of an $\EE^{\sstar m}$-configuration into $\Mon^n$:
\begin{quotation}
 $(\dagger)_{m,n}$: There exists a function $f : \omega^m \rightarrow \Mon^n$ such that, for all $t, t' \in {}^m 2$, for all $\oa, \ob, \oa', \ob' \in \omega^m$ with $\oa \le_t \ob$ and $\oa' \le_{t'} \ob'$,
 \[
  \tp_L(f(\oa), f(\ob)) = \tp_L(f(\oa'), f(\ob')) \mathiff t = t'.
 \]
\end{quotation}
In particular, when $t = t'$, $f(\oa)_\ell = f(\ob)_\ell$ if and only if $f(\oa')_\ell = f(\ob')_\ell$ for each $\ell < n$.  Moreover, by choosing the constantly zero function for $t$, we see that the function $\oa \mapsto \tp_L(f(\oa))$ is constant.  Restricting to $2^m \subseteq \omega^m$, we see that, for all $\oa, \ob \in 2^m$,
\begin{equation}\label{Eq_2mtypes}
 \tp_L(f(\ozero), f(\oa)) = \tp_L(f(\ozero), f(\ob)) \mathiff \oa = \ob.
\end{equation}

As advertised, we get the following lemma.

\begin{lem}\label{Lem_EEUniformize}
 If there exists an $\EE^{\sstar m}$-configuration into $\Mon^n$, then $(\dagger)_{m,n}$ holds.
\end{lem}

\begin{proof}
 Let $(I, f_A)_{A \in \EE^{\sstar m}}$ be a $\EE^{\sstar m}$-configuration into $\Mon^n$ over $C \subseteq \Mon$ finite.  Fix $k < \omega$.  By Lemma \ref{Lem_ColorBoxes2}, there exists $n < \omega$ such that, for all $c : \binom{n^m}{\le 2} \rightarrow S^L_{2n}(C)$, there exist $Y_0, \dots, Y_{m-1} \in \binom{n}{k}$ such that, for all $t \in {}^m 2$, $c$ is constant on
 \[
  X_t = \left\{ \{ \oa, \ob \} : \oa, \ob \in \prod_{i < m} Y_i, \oa \le_t \ob \right\}.
 \]
 Let $A = n^{m+1}$ be the $L_m$-structure given in Lemma \ref{Lem_EEmuniverse}.  In particular, this holds for the coloring $c$ given by, for all $\oa, \ob \in n^m$ with $\oa \le_\lex \ob$,
 \[
  c(\{ \oa, \ob \}) = \tp_L( f_A(\oa, 0), f_A(\ob, 0) / C).
 \]
 As $k$ was arbitrary, by compactness, there exists $f : \omega^m \rightarrow \Mon^n$ such that, for all $\oa, \ob \in \omega^m$ and $i < m$, $\Mon \models I(E_i)(f(\oa), f(\ob))$ if and only if $a_i = b_i$ and, for all $t \in {}^m 2$, the function $(\oa, \ob) \mapsto \tp_L(f(\oa), f(\ob) / C)$ is constant for all $\oa, \ob \in \omega^m$ with $\oa \le_t \ob$.  In particular, the function $\oa \mapsto \tp_L(f(\oa) / C)$ is constant.  Since $L$ is a binary language, the type $\tp_L(f(\oa), f(\ob) / C)$ is determined by the type $\tp_L(f(\oa), f(\ob))$.  Therefore, $f$ witnesses that $(\dagger)_{m,n}$ holds.
 \if\papermode0
 { \color{violet}

  \textit{More Details}: Let $\Sigma$ be the type
  \begin{align*}
   \Sigma( \oy_{\oa} )_{\oa \in \omega^m} = & \ \{ \varphi_i(\oy_{\oa}, \oy_{\ob})^{\condiff a_i = b_i} : \oa, \ob \in \omega^m \} \cup \\ & \ \bigcup_{t \in {}^m 2} \{ \psi(\oy_{\oa}, \oy_{\ob}) \leftrightarrow \psi(\oy_{\oa'}, \oy_{\ob'}) : \\ & \ \oa, \ob, \oa', \ob' \in \omega^m, \oa \le_t \ob, \oa' \le_t \ob', \psi(\oy) \in L(C) \}.
  \end{align*}
  Let $\Sigma_0 \subseteq \Sigma$ be finite.  Let $A \subseteq \omega$ be finite such that $\Sigma_0$ mentions only variables of the form $\oy_{\oa}$ for $\oa \in A^m$.  Then, for $k = |A|$ (and $Y_m = \{ 0 \}$), $g \left( \prod_{i \le m} Y_i \right)$ is a witness to the consistency of $\Sigma_0$.  By compactness, $\Sigma$ is consistent, say realized by $( \oc_{\oa} )_{\oa \in \omega^m}$.  Let $f : \omega^m \rightarrow \Mon^n$ be given by $f(\oa) = \oc_{\oa}$.
  
  Clearly, by construction, for all $t \in {}^m 2$ and $\oa, \ob, \oa', \ob' \in \omega^m$ with $\oa \le_t \ob$ and $\oa' \le_t \ob'$, $\tp_L(f(\oa), f(\ob)) = \tp_L(f(\oa'), f(\ob'))$.  On the other hand, if $t, t' \in {}^m 2$ are distinct, then there exists $i < m$ such that $t(i) \neq t'(i)$.  Without loss of generality, say $t(i) = 0$ and $t'(i) = 1$.  Fix $\oa, \ob, \oa', \ob' \in \omega^m$ with $\oa \le_t \ob$ and $\oa' \le_{t'} \ob'$.  Then,
  \[
   \Mon \models \varphi_i(f(\oa), f(\ob)) \wedge \neg \varphi_i(f(\oa'), f(\ob')).
  \]
  Therefore, $\tp_L(f(\oa), f(\ob) / C) \neq \tp_L(f(\oa'), f(\ob') / C)$.  Since $\tp_L(f(\oa) / C) = \tp_L(f(\oa') / C)$ and $\tp_L(f(\ob) / C) = \tp_L(f(\ob') / C)$ and $L$ is binary, we conclude that
  \[
   \tp_L(f(\oa), f(\ob)) \neq \tp_L(f(\oa'), f(\ob'))
  \]
 }
 \fi
\end{proof}

Suppose $(\dagger)_{m,n}$ holds for some $n, m < \omega$, witnessed by $f$.  Let $\ee_i$ be the $i$th standard basis vector.  For each $\ell < n$, let
\[
 V_\ell = \left\{ \oa \in \omega^m : f(\ozero)_\ell = f(\oa)_\ell \right\}.
\]

\begin{lem}\label{Lem_EESubspace}
 There exists $I_\ell \subseteq m$ such that
 \[
  V_\ell = \{ \oa \in \omega^m : (\forall i \in m \setminus I_\ell)[ a_i = 0 ] \}
 \]
 In other words, $V_\ell$ is the $\omega$-span of $\{ \ee_i : i \in I_\ell \}$.
\end{lem}

\begin{proof}
 Let $I_\ell = \{ i < m : (\exists \oa \in V_\ell)[ a_i > 0 ] \}$.  We show this works.

 Clearly $\ozero \in V_\ell$.  Fix $\oa \in V_\ell$ non-zero and $i < m$ such that $a_i > 0$.  Notice that, for all $t \in {}^m 2$,
 \[
  \ozero \le_t \oa \mathiff \ozero \le_t \oa + \ee_i.
 \]
 By $(\dagger)_{m,n}$, since $f(\ozero)_\ell = f(\oa)_\ell$, $f(\ozero)_\ell = f(\oa + \ee_i)_\ell$.  Thus, $f(\oa)_\ell = f(\oa + \ee_i)_\ell$.  By $(\dagger)_{m,n}$, $f(\ozero)_\ell = f(\ee_i)_\ell$.  Thus, $\ee_i \in V_\ell$.
 
 Suppose that $\oa, \ob \in V_\ell$.  Then,
 \[
  f(\ozero)_\ell = f(\oa)_\ell \text{ and } f(\ozero)_\ell = f(\ob)_\ell.
 \]
 By $(\dagger)_{m,n}$, $f(\oa)_\ell = f(\oa + \ob)_\ell$.  Therefore, $\oa + \ob \in V_\ell$.
 
 Putting these facts together, we get the desired conclusion. 
\end{proof}

\begin{lem}\label{Lem_EEEqual}
 Suppose $(\dagger)_{m,n}$ holds, witnessed by $f$.  For all $\ell, \ell' < n$ and $\oa \in \omega^m$, if $f(\ozero)_\ell = f(\oa)_{\ell'}$, then $\oa \in V_\ell \cap V_{\ell'}$.
\end{lem}

\begin{proof}
 By $(\dagger)_{m,n}$, $f(\ozero)_\ell = f(2 \oa)_{\ell'}$, hence $f(\oa)_{\ell'} = f(2 \oa)_{\ell'}$.  By $(\dagger)_{m,n}$, $f(\ozero)_{\ell'} = f(\oa)_{\ell'}$, hence $\oa \in V_{\ell'}$.  On the other hand, by $(\dagger)_{m,n}$, $f(\oa)_\ell = f(2 \oa)_{\ell'}$, hence $f(\ozero)_\ell = f(\oa)_\ell$.  Thus, $\oa \in V_\ell$.
\end{proof}

We are now ready to compute $\Rk_\EE(n)$.

\begin{expl}[$\KK = \EE$]\label{Expl_ERandGr}
 For all $n \ge 1$,
 \[
  \Rk_\EE(n) = \begin{cases} 1 & \text{ if } n = 1, \\ n^2 - 1 & \text{ if } n \ge 2 \end{cases}.
 \]
\end{expl}

\begin{proof}
 Since $\EE$ is reflexive and symmetric, Proposition \ref{Prop_KKquad1} says that $\Rk_\EE(n) \ge n^2 - 1$.  Moreover, Proposition \ref{Prop_EELowerBound} says that $\Rk_\EE(1) \ge 1$.
 
 Towards a contradiction, suppose $\Rk_\EE(1) \ge 2$; hence, $(\dagger)_{2,1}$ holds, say witnessed by $f : \omega^2 \rightarrow \Mon$.  By \eqref{Eq_2mtypes}, this gives us at least four distinct $2$-$L$-types over $\emptyset$.  On the other hand, there are only three such types, a contradiction.
 
 So it suffices to show that $\Rk_\EE(n) < n^2$ when $n \ge 2$.  To accomplish this, we prove, by induction on $n$, that $(\dagger)_{n^2,n}$ fails.
 
 We will deal with the base case of $n = 2$ at the end.  Fix $n \ge 3$ and assume that $(\dagger)_{(n-1)^2,n-1}$ fails.  Towards a contradiction, suppose that $(\dagger)_{n^2,n}$ holds, say witnessed by $f : \omega^{n^2} \rightarrow \Mon^n$.
 
 \begin{claim*}
  For all $\ell < n$, $| I_\ell | < (n-1)^2$ (where $I_\ell$ is as defined in Lemma \ref{Lem_EESubspace}, for this choice of $f$).
 \end{claim*}
 
 \begin{proof}[Proof of Claim]
  Fix $\ell < n$.  Towards a contradiction, suppose $| I_\ell | \ge (n-1)^2$.  Let $m = (n-1)^2$, let $\sigma : m \rightarrow I_\ell$ be any injective function, and, for each $\oa \in \omega^m$, let $\oa_\sigma \in \omega^{n^2}$ be given by
  \[
   a_{\sigma,i} = \begin{cases} a_j & \text{ if } i = \sigma(j), \\ 0 & \text{ if } i \notin \image(\sigma) \end{cases}.
  \]
  In particular, $\oa_\sigma \in V_\ell$.  Hence, for all $\oa, \ob \in \omega^m$, $f(\oa_\sigma)_\ell = f(\ob_\sigma)_\ell$.  Define $f' : \omega^m \rightarrow \Mon^{n-1}$ as follows: For each $\oa \in \omega^m$, let $f'(\oa) = f(\oa_\sigma)$ restricted to exclude the $\ell$th coordinate.  It is easy to check that $f'$ satisfies $(\dagger)_{m,n-1}$, contrary to the inductive hypothesis.
  \if\papermode0
  { \color{violet}
   
   \textit{More Details:} Fix $t, t' \in {}^m 2$ and $\oa, \ob, \oa', \ob' \in \omega^m$ with $\oa \le_t \ob$ and $\oa' \le_{t'} \ob'$.  Then, $\oa_\sigma$ and $\ob_\sigma$ are ``in the same direction'' as $\oa'_\sigma$ and $\ob'_\sigma$ if and only if $t = t'$.  Hence, $\tp_L(f(\oa_\sigma), f(\ob_\sigma)) = \tp_L(f(\oa'_\sigma), f(\ob'_\sigma))$ if and only if $t = t'$.  On the other hand, $f(\oa_\sigma)_\ell = f(\ob_\sigma)_\ell$ and $f(\oa'_\sigma)_\ell = f(\ob'_\sigma)_\ell$.  Therefore, $\tp_L(f'(\oa), f'(\ob)) = \tp_L(f'(\oa'), f'(\ob'))$ if and only if $t = t'$.
  }
  \fi
 \end{proof}
 
 Let $m = n^2$ and let
 \[
  V = \left\{ \oa \in 2^m : (\exists \ell < n)(\forall i \in m \setminus I_\ell)[ a_i = 0 ] \right\}.
 \]
 In other words, $V$ is the union of $2^m \cap V_\ell$ over all $\ell < n$.  By the claim, for each $\ell < n$, $|I_\ell| \le (n-1)^2 - 1 = n^2 - 2n$.  Thus,
 \[
  \left| 2^m \cap V_\ell \right| \le 2^{n^2 - 2n}.
 \]
 Therefore,
 \[
  \left| 2^m \setminus V \right| \ge 2^{n^2} - n2^{n^2 - 2n}.
 \]
 
 \begin{claim*}
  $2^{n^2} - n2^{n^2 - 2n} > 2^{n^2 - 1} + 2^{\binom{n+1}{2} - 1}$.
 \end{claim*}
 
 \begin{proof}[Proof of Claim]
  Since $n \ge 3$, $(n+1)(n-1) > \frac{1}{2} n (n+1)$.  Thus, $n^2 - 2 > \binom{n+1}{2} - 1$.  So
  \[
   2^{n^2 - 2} > 2^{\binom{n+1}{2} - 1}.
  \]
  Similarly, $(n+1)(n-1) > n(n-1)$.  Thus, $n^2 - 2 > n^2 - n - 1$, so
  \[
   2^{n^2 - 2} > 2^{n-1} 2^{n^2 - 2n}.
  \]
  Since $n \ge 3$, $n < 2^{n-1}$.  Therefore,
  \[
   2^{n^2 - 2} > n 2^{n^2 - 2n}.
  \]
  Putting these together, we get
  \[
   2^{n^2 - 1} > 2^{\binom{n+1}{2} - 1} + n 2^{n^2 - 2n}.
  \]
  This gives us the desired conclusion.
 \end{proof}
 
 Therefore, by \eqref{Eq_2mtypes},
 \[
  \left| \left\{ \tp_L(f(\ozero), f(\oa)) : \oa \in 2^m \setminus V \right\} \right| = | 2^m \setminus V | > | \cGs^n | + \frac{1}{2} | \cGns^n |.
 \]
 However, by Lemma \ref{Lem_EEEqual}, for each $\oa \in 2^m \setminus V$ and all $\ell, \ell' < n$, $f(\ozero)_\ell \neq f(\oa)_{\ell'}$.  Hence, each type $\tp_L(f(\ozero), f(\oa))$ corresponds to a unique element of $\cG^n$ as in the proof of Proposition \ref{Prop_KKquad2}.  Since $\EE^{\sstar m}$ is symmetric, as in the proof of Proposition \ref{Prop_KKquad3}, we conclude that there are at most $| \cGs^n | + \frac{1}{2} | \cGns^n |$ such types.  This is a contradiction.
 
  For the base case, towards a contradiction, suppose that $(\dagger)_{4,2}$ holds.  This argument proceeds similarly to the general inductive argument.  Notice that, for all $\ell < 2$, $|I_\ell| \le 1$ (this follows from a similar argument to the first claim).  Thus, $| 2^4 \setminus V | \ge 2^4 - 3 = 13$.  On the other hand,
 \[
  | \cG^2 | + \frac{1}{2} | \cGns^2 | = 12.
 \]
\end{proof}

%%%%%%%%%%%%%%%%%%%%%%%%%%%%%%%%%%%%%%%%%%%%%%%%%%%%%%%
%% Subsection -- Ranks and the Independence Property %%
%%%%%%%%%%%%%%%%%%%%%%%%%%%%%%%%%%%%%%%%%%%%%%%%%%%%%%%

\subsection{Ranks and the Independence Property}\label{Subsect_RankDivide}

How do the ranks studied above interact with model-theoretic dividing lines (in particular, the independence property)?

As long as $\LO$-rank is finite, $\LO$-rank grows linearly if $T$ has NIP and grows quadratically if $T$ has the independence property.

\begin{thm}\label{Thm_NIPLinearQuadLO}
 Let $T$ be any complete first-order theory such that $\Rk_\LO(1) < \infty$.
 \begin{enumerate}
  \item If $T$ has NIP, then there exists $C \in \mathbb{R}$ such that, for all $n \ge 1$,
   \[
    \Rk_\LO(n) \le Cn.
   \]
  \item If $T$ has the independence property, then there exists $C \in \mathbb{R}$ such that, for sufficiently large $n$,
   \[
    \Rk_\LO(n) \ge Cn^2.
   \]
 \end{enumerate}
\end{thm}

\begin{proof}
 (1): As noted in Example \ref{Expl_LONIP}, if $T$ has NIP, then $\Rk_\LO$ is additive.  Thus, if we let $C = \Rk_\LO(1)$, $\Rk_\LO(n) = Cn$ for all $n \ge 1$.
 
 (2): Assume that $T$ has the independence property and $\Mon$ is a monster model for $T$.  By Theorem \ref{Thm_CCExamples} (2) and Lemma \ref{Lem_InjectiveConfig}, there exists an injective $\GG$-configuration into $\Mon^k$ for some $k < \omega$.  Moreover, by Proposition \ref{Prop_KKquad1}, for all $m < \omega$, there exists a parameter-free $\LO^{\sstar (m^2 - 1)}$-configuration into $\Mon_1^m$, where $\Mon_1$ is a monster model for the theory of the \fraisse limit of $\GG$.  By Proposition \ref{Prop_ComposeConfig}, there exists an $\LO^{\sstar (m^2 - 1)}$-configuration into $\Mon^{km}$.  Therefore,
 \[
  \Rk_\LO(km) \ge m^2 - 1.
 \]
 For $n \ge k$, let $m = \lfloor n / k \rfloor$.  Then, by Lemma \ref{Lem_RkKKOrder},
 \[
  \Rk_\LO(n) \ge \Rk_\LO(km) \ge m^2 - 1 \ge \frac{1}{k^2} n^2 - \frac{2}{k} n.
 \]
\end{proof}

In the next few examples, we examine the applicability of the preceding theorem.

\begin{expl}\label{Expl_RKLO10}
 Let $T$ be a complete theory in some language $L$.  Suppose that all indiscernible sequences of singletons are set indiscernible.  By Remark \ref{Rem_GenIndisc}, we can make any $\LO$-configuration into an indiscernible one, so there exists no $\LO$-configuration into singletons of $T$.  Thus, $\Rk_\LO(1) = 0$.  Therefore, for any such theory $T$, $T$ has NIP if and only if there exists $C \in \mathbb{R}$ such that, for all $n \ge 1$, $\Rk_\LO(n) \le Cn$.  This applies, for example, to the theory of the \fraisse limit of any irreflexive, symmetric \fraisse class in a finite relational language.
\end{expl}

\begin{expl}
 Let $L$ consist of infinitely many binary relation symbols $<_i$ for $i < \omega$ and let $T$ be the model companion of the theory which says each $<_i$ is a linear order.  Clearly $T$ has quantifier elimination, so $T$ has NIP.  Thus, $\LO$-rank and op-dimension coincide.  Therefore, $\Rk_\LO(1) = \infty$.  So, there are examples of NIP theories for which the theorem does not apply.  If we replace ``linear order'' with ``partial order'' in the definition of $T$, we obtain an example of a theory with the independence property such that $\Rk_\LO(1) = \infty$.
\end{expl}

Similar to $\LO$-rank, if the target theory has the independence property and $\EE$-rank is finite, then $\EE$-rank necessarily grows quadratically.

\begin{prop}\label{Prop_EERankNotAdditive}
 Let $T$ be any complete first-order theory with the independence property such that $\Rk_\EE(1) < \infty$.  Then, there exists $C \in \mathbb{R}$ such that, for sufficiently large $n$,
 \[
  \Rk_\EE(n) \ge Cn^2.
 \]
\end{prop}

\begin{proof}
 This is similar to the proof of Theorem \ref{Thm_NIPLinearQuadLO} (2).
 \if\papermode0
 { \color{violet}
 
  \emph{More Detail:} By Theorem \ref{Thm_CCExamples} (2), there exists a $\GG$-configuration into $\Mon^k$ for some $k < \omega$.  By Example \ref{Expl_ERandGr}, for each $n \ge 2$, there exists a parameter-free $\EE^{\sstar (n^2-1)}$ configuration into $\Mon_1^n$, where $\Mon_1$ is the monster model for the theory of the \fraisse limit of $\GG$.  By Proposition \ref{Prop_ComposeConfig}, there exists an $\LO^{\sstar (n^2-1)}$-configuration into $\Mon^{kn}$.  Therefore,
 \[
  \Rk_\EE(kn) \ge n^2 - 1.
 \]
 By Lemma \ref{Lem_RkKKOrder}, for all $m \ge 2k$, let $n = \lfloor m / k \rfloor$.  Then,
 \[
  \Rk_\EE(m) \ge \Rk_\EE(kn) \ge n^2 - 1 \ge \frac{1}{k^2} m^2 - \frac{2}{k} m.
 \]
 }
 \fi
\end{proof}

Similarly, in any theory with the independence property, so long as $\GG$-rank is finite, $\GG$-rank grows quadratically.

\begin{prop}\label{Prop_GGRankNotAdditive}
 If $T$ is any theory and $\pi$ any partial type with $\Rk_\GG(\pi) \ge 1$, then
 \[
  \Rk_\GG(\pi^{\times n}) \ge n^2 - 1.
 \]
\end{prop}

\begin{proof}
 This is similar to the proof of Theorem \ref{Thm_NIPLinearQuadLO} (2).
 \if\papermode0
 { \color{violet}
 
  \emph{More Detail:} By Proposition \ref{Prop_KKquad1}, there exists a parameter-free $\GG^{\sstar (n^2 - 1)}$-configuration into $\Mon_1^n$, where $\Mon_1$ is the monster model for the theory of the \fraisse limit of $\GG$.  By assumption, there exists a $\GG$-configuration into $\pi$.  By Proposition \ref{Prop_ComposeConfig}, there exists a $\GG^{\sstar (n^2 - 1)}$-configuration into $\pi^{\times n}$.  
 }
 \fi
\end{proof}

From Theorem \ref{Thm_NIPLinearQuadLO}, Proposition \ref{Prop_EERankNotAdditive}, and Proposition \ref{Prop_GGRankNotAdditive}, we obtain the following result about additivity of ranks: If $T$ has the independence property, $\KK = \EE$, $\LO$, or $\GG$, and $\KK$-rank is finite, then $\KK$-rank is not additive.

%%%%%%%%%%%%%%%%%
%% FUTURE WORK %%
%%%%%%%%%%%%%%%%%

\section{Future Work}\label{Sect_Conclusion}

From the study of $\KK$-configurations and $\KK$-rank in the previous sections, we are left with a few open questions.

\begin{ques}
 Under what conditions is $\KK$-rank a generalization of a known rank in model theory?
\end{ques}

In Example \ref{Expl_LONIP}, we establish that $\LO$-rank coincides with op-dimension when $T$ has NIP, which implies that $\LO$-rank coincides with dp-rank when $T$ is distal.  On the other hand, $\LO$-rank diverges from op-dimension when $T$ has the independence property.  Similarly, $\EE$-rank appears to be related to dp-rank, but the exact relationship remains unclear.  Proposition \ref{Prop_EEUpperBound} establishes that dp-rank is an upper bound for $\EE$-rank while Corollary \ref{Cor_EEDpRank} shows that these ranks coincide on $\Mon^n$ when $T$ is dp-minimal.  On the other hand, even in NIP theories, dp-rank and $\EE$-rank diverge, as shown in Example \ref{Expl_EENotDp}.  This example is distal, however, which leads to an interesting question

\begin{ques}
 Do dp-rank and $\EE$-rank coincide for stable theories?
\end{ques}

Along a similar line, when is $\KK$-rank additive (Question \ref{Ques_KKRankAdd})?  We see that $\LO$-rank, and even $\GG$-rank (trivially), are additive when $T$ has NIP.  On the other hand, these ranks fail additivity when moving to theories with the independence property.  Is it possible that, more generally, $\KK$-ranks are additive on NIP theories?  In particular, is $\EE$-rank additive on NIP theories?

Although we examined a few examples of \algtriv \fraisse classes in this paper, there are other classes that are currently unexplored.  We have one result on $\TT$-rank, Example \ref{Expl_FFandGr}, and no results on $\HH_k$-rank for $k > 2$.  It is possible, for example, that $\TT$-rank coincides with $\LO$-rank for some types.  Most of the technology developed in this paper relies on the index language being binary, which makes analyzing $\HH_k$-rank more challenging when $k > 2$.  In future work, we would like to examine $\KK$-rank for these and other classes, $\KK$.

Finally, Section \ref{Sect_DividingLines} and the relationships to \cite{gh} reveal other interesting open questions.  For example, we have the strict $\lesssim$-chain
\[
 \EE < \LO < \GG < \HH_3 < \HH_4 < \dots.
\]
This leads to the following question

\begin{ques}
 Is $\lesssim$ a linear quasi-order on \algtriv \fraisse classes?  In particular, are there any classes strictly between $\EE$ and $\LO$?
\end{ques}

In other words, is there a non-trivial dividing line, in the sense of $\cC_\KK$, below stability?

\appendix

%%%%%%%%%%%%%%%%%%%%%%%%%%%%%%%%%%%%%%
%% Appendix A: Combinatorial Lemmas %%
%%%%%%%%%%%%%%%%%%%%%%%%%%%%%%%%%%%%%%

\section{Combinatorial Lemmas}\label{App_CombLemmas}

Fix $k < \omega$ and let
\[
 \cD_k = \{ t \in {}^k \{ -1, 0, 1 \} : t(i) = 1 \text{ for } i \text{ minimal such that } t(i) \neq 0 \}.
\]
For $\oa, \ob \in \omega^k$ and $t \in \cD_k$, define
\[
 \oa \le_t \ob \text{ if, for all } i < k,
 \begin{cases}
  a_i < b_i & \text{ if } t(i) = 1, \\
  a_i = b_i & \text{ if } t(i) = 0, \\
  a_i > b_i & \text{ if } t(i) = -1.
 \end{cases}
\]
Finally, for all $\oa, \ob \in \omega^k$, define
\[
 \oa \le_\lex \ob \text{ if } a_i < b_i \text{ for } i \text{ minimal such that } a_i \neq b_i.
\]
Note that $\oa \le_\lex \ob$ if and only if there exists $t \in \cD_k$ such that $\oa \le_t \ob$.

\begin{lem}\label{Lem_ColorBoxes2}
 For all $k, \ell, m < \omega$, there exists $n < \omega$ such that, for all colorings $c : \binom{n^k}{\le 2} \rightarrow \ell$, there exist $Y_0, \dots, Y_{k-1} \in \binom{n}{m}$ such that, for all $t \in \cD_k$, $c$ is constant on the set
 \[
  X_t = \left\{ \{ \oa, \ob \} : \oa, \ob \in \prod_{i < k} Y_i, \oa \le_t \ob \right\}.
 \]
\end{lem}

\begin{proof}
 By induction on $k$.  Let $k = 1$ and fix $\ell, m < \omega$.  By Ramsey's Theorem, there exists $n$ such that, for all colorings $c : \binom{n}{\le 2} \rightarrow \ell$, there exists $Y \in \binom{n}{m}$ such that $c$ is constant on $\binom{Y}{1}$ and $c$ is constant on $\binom{Y}{2}$.  Since $X_0 = \binom{Y}{1}$ and $X_{1} = \binom{Y}{2}$, this is the desired conclusion.
 
 Fix $k, m, \ell < \omega$.  Let
 \[
  \ell' = {}^{\cD_k \times \{ -1, 1 \}} \ell.
 \]
 By Ramsey's Theorem, there exists $n' < \omega$ such that, for all colorings $c' : \binom{n'}{\le 2} \rightarrow \ell'$, there exists $Y_k \in \binom{n'}{m}$ such that $c'$ is constant on $\binom{Y_k}{1}$ and $c'$ is constant on $\binom{Y_k}{2}$.  Let
 \[
  \ell'' = {}^{(n')^2} \ell.
 \]
 By the inductive hypothesis, there exists $n'' < \omega$ such that, for all colorings $c'' : \binom{(n'')^k}{\le 2} \rightarrow \ell''$, there exist $Y_0, \dots, Y_{k-1} \in \binom{n''}{m}$ such that, for all $t \in \cD_k$, $c''$ is constant on $X_t$.  Let $n = \max \{ n', n'' \}$.
 
 Fix a coloring $c : \binom{n^{k+1}}{\le 2} \rightarrow \ell$.  This induces a coloring $c'' : \binom{(n'')^k}{\le 2} \rightarrow \ell''$ given by: for each $\oa, \ob \in (n'')^k$ with $\oa \le_\lex \ob$, for each $i, j \in n'$, let
 \[
  c''( \{ \oa, \ob \} )(i, j) = c( \{ \oa \concat i, \ob \concat j \} ).
 \]
 
 Thus, there exist $Y_0, \dots, Y_{k-1} \in \binom{n''}{m}$ such that, for all $t \in \cD_k$, $c''$ is constant on $X_t$.  Now define $c' : \binom{n'}{\le 2} \rightarrow \ell'$ as follows: for each $i \le j < n'$, $t \in \cD_k$, and $s \in \{ -1, 1 \}$, choose $\oa, \ob \in \prod_{i < k} Y_i$ with $\oa \le_t \ob$ and set
 \[
  c'( \{ i, j \} )(t,s) = \begin{cases} c( \{ \oa \concat i, \ob \concat j \} ) & \text{ if } s = 1, \\ c( \{ \oa \concat j, \ob \concat i \} ) & \text{ if } s = -1. \end{cases}
 \]
 Since $c''$ is constant on $X_t$ for each $t$, this function is independent of the choice of $\oa$ and $\ob$.  Thus, there exists $Y_k \in \binom{n'}{m}$ such that $c'$ is constant on $\binom{Y_k}{1}$ and $c'$ is constant on $\binom{Y_k}{2}$.  We claim that $Y_0, \dots, Y_k$ work for $c$.
 
 Fix $t \in \cD_{k+1}$.  If $t(k)=0$, let
 \[
  r = c'( \{ i \} )(t |_k, 1 )
 \]
 for any choice of $i \in Y_k$.  Since $c'$ is constant on $\binom{Y_k}{1}$, this is independent of the choice of $i$.  If $t(k) \neq 0$, let
 \[
  r = c'( \{ i, j \})(t |_k, t(k))
 \]
 for any choice of $i, j \in Y_k$ with $i < j$.  Since $c'$ is constant on $\binom{Y_k}{2}$, this is independent of the choice of $i$ and $j$.  Then, for any $\oa, \ob \in \prod_{i \le k} Y_i$ such that $\oa \le_t \ob$, we have that
 \[
  c( \{ \oa, \ob \} ) = r.
 \]
 This is what we wanted to prove.
\end{proof}

\begin{cor}\label{Cor_ColorBoxes}
 For all $k, \ell, m < \omega$, there exists $n < \omega$ such that, for all colorings $c : n^k \rightarrow \ell$, there exist $Y_0, \dots, Y_{k-1} \in \binom{n}{m}$ such that $c$ is constant on $\prod_{i < k} Y_i$.
\end{cor}

\begin{proof}
 Since any coloring $c : n^k \rightarrow \ell$ can be extended arbitrarily to a coloring $c : \binom{n^k}{\le 2} \rightarrow \ell$, this follows immediately from Lemma \ref{Lem_ColorBoxes2}.
 \if\papermode0
 { \color{violet}
 
  \emph{More Detail:} We prove this by induction on $k$.  When $k = 1$, let $n = (m-1) \ell + 1$ and the conclusion follows by Pigeonhole Principle.
 
 Assume this holds for $k$ and $n$.  Let $n' = (m-1) \ell^{n^k} + 1$.  Fix a coloring $c : (n')^{k+1} \rightarrow \ell$.  This gives a coloring $c^* : n' \rightarrow {}^{n^k} \ell$ via $c^*(i)(\oj) = c(i,\oj)$.  By Pigeonhole principle, there exists $Y_k \in \binom{n'}{m}$ such that $c^*$ is constant on $Y_k$.  Say that the constant value is $s$.  Then, $s : n^k \rightarrow \ell$ is a coloring.  By the induction hypothesis, there exist $Y_0, \dots, Y_{k-1} \in \binom{n}{m}$ such that $s$ is constant on $\prod_{i < k} Y_i$.  Therefore, $c$ is constant on $\prod_{i \le k} Y_i$.
 }
 \fi
\end{proof}

\end{document}